\documentclass[12pt]{article}
\usepackage{amsmath, latexsym, amsfonts, amssymb, amsthm, amscd}
\usepackage[left=3cm, right=2.5cm, top=2.5cm, bottom=2.5cm]{geometry}
\usepackage{color}

\newtheorem{proposition}{Proposition}
\newtheorem{lemma}{Lemma}
\newtheorem{rem}{Remark}

\newtheorem{theorem}{Theorem}

\newcommand{\p}{\mathbb{P}}
\newcommand{\e}{\mathbb{E}}
\newcommand{\ud}{\mathrm{d}}

\newcommand{\R}{\mathbb{R}}

\newcommand{\F}{\mathcal{F}}

\newcommand{\Ind}[1]{\mathbf{1}_{\{#1\}}}
\newcommand{\Expo}[1]{\exp\left\{#1\right\}}
\newcommand{\Exp}[1]{\mathbb{E}\left[#1\right]}
\newcommand{\Procon}[2]{\mathbb{P}\left(#1\big| #2 \right)}
\newcommand{\Expx}[2]{\mathbb{E}_{#1}\left[ #2 \right]}
\newcommand{\Prox}[2]{\mathbb{P}_{#1}\left( #2 \right)}
\newcommand{\Expconx}[3]{\mathbb{E}_{#1}\left[#2\big| #3 \right]}
\newcommand{\Proconx}[3]{\mathbb{P}_{#1}\left(#2\big| #3 \right)}
\newcommand{\pro}{\mathbb{P}}

\begin{document}

\title{Continuous state branching processes in  random  environment: The Brownian case.}
\author{ S. Palau\footnote{ {\sc Centro de Investigaci\'on en Matem\'aticas A.C. Calle Jalisco s/n. 36240 Guanajuato, M\'exico.} E-mail: sandra.palau@cimat.mx. Corresponding author}\,\, and J.C. Pardo\footnote{ {\sc Centro de Investigaci\'on en Matem\'aticas A.C. Calle Jalisco s/n. 36240 Guanajuato, M\'exico.} E-mail: jcpardo@cimat.mx}
}
\maketitle
\vspace{0.2in}

\begin{abstract} 
\noindent We consider continuous state branching processes that are perturbed by a Brownian motion. These processes are constructed as the unique strong solution of a stochastic differential equation. The long-term extinction and explosion behaviours are studied. In the stable case, the extinction and explosion probabilities are given explicitly. We find three regimes for the asymptotic behaviour of the explosion probability and, as in the case of branching processes in random environment, we find five regimes for the asymptotic behaviour of the extinction probability. In the supercritical regime, we study the process conditioned on eventual extinction where three regimes for the asymptotic  behaviour of the extinction probability appear. Finally, the process conditioned on non-extinction and the process with immigration are given.

\bigskip

\noindent {\sc Key words and phrases}: Continuous state branching processes in  random environment, Brownian motion, explosion and extinction probabilities, exponential functional of Brownian motion, Q-process, supercritical  process conditioned on eventual extinction, continuous state branching processes with immigration in  random environment.

\bigskip

\noindent MSC 2000 subject classifications: 60G17, 60G51, 60G80.
\end{abstract}

\vspace{0.5cm}

\section{Introduction}
A $[0,\infty]$-valued strong Markov process $Y =(Y_t, t\geq 0)$  with probabilities $(\mathbb{P}_x, \, x\geq 0)$ is called a {\it continuous-state branching process } (CB-processes for short) if it has paths that are right-continuous with left limits and its law observes the branching
property; 
 i.e. for any $x,y\geq 0$, $\mathbb{P}_{x+y}$ is equal in law to the convolution of $\mathbb{P}_x$ and $\mathbb{P}_y$. CB-processes may be thought of as the continuous (in time and space) analogues of classical Bienaym\'e-Galton-Watson  processes.  CB-processes have  been introduced by
Jirina \cite{ji} and studied by many authors including Bingham
\cite{bi}, Grey \cite{gre}, Lamperti \cite{la1}, to name but a few.  
The branching property implies  that the 
Laplace transform  of $Y_t$ satisfies
\begin{equation}\label{CBut}
\mathbb{E}_x\Big[ e^{-\lambda  Y_t}\Big]=\exp\{-x
u_t(\lambda)\},\qquad\textrm{for }\lambda\geq 0,
\end{equation}
for some function $u_t(\lambda)$.  According to Silverstein
\cite{si}, the function $u_t(\lambda):[0,\infty)\to[0,\infty]$ solves  the integral
equation
\begin{equation}\label{ut}
u_t(\lambda)+\int_0^t \psi(u_s(\lambda)){\rm d} s=\lambda,
\end{equation}
 where $\psi$ satisfies the celebrated L\'evy-Khintchine formula, i.e.
\begin{equation*}
\psi(\lambda)=-q-a\lambda+\gamma^2
\lambda^2+\int_{(0,\infty)}\big(e^{-\lambda x}-1+\lambda x{\mathbf 1}_{\{x<1\}}\big)\mu(\ud x),
\end{equation*}
where $q\ge 0$, $a\in \mathbb{R}$, $\gamma\geq 0$ and $\mu$ is a measure concentrated on $(0,\infty)$ such that
$\int_{(0,\infty)}\big(1\land x^2\big)\mu(\ud x)$ is finite. The function $\psi:[0,\infty)\to (-\infty, \infty)$ is convex and is known as the branching mechanism of
$Y$.

Observe that  $0$ and $\infty$ are two absorbing states. In other words, let
\[
T_0=\inf\{t\ge 0: Y_t=0\} \qquad \textrm{ and } \qquad T_\infty=\inf\{t\ge 0: Y_t=\infty\}
\]
denote the extinction and explosion times, respectively. Then $Y_t=0$ for every $t\ge T_0$ and $Y_t=\infty$ for every $t\ge T_\infty$.
More precisely, let   $\eta$ be the largest root of the branching mechanism $\psi$,  i.e. $\eta=\sup\{\theta \ge 0: \psi(\theta)=0\}$, (with the convention that $\sup\{\emptyset\}=\infty$). Then for every $x>0$:
\begin{itemize}
	\item[i)] if $\eta=\infty$ or if $\int^\infty  {\rm d}\theta/ \psi(\theta)=\infty$, we have $\p_x(T_0<\infty)=0$,
	\item[ii)] if $\eta<\infty$ and  $\int^\infty  {\rm d}\theta/ \psi(\theta)<\infty$, we define
	\[
	\phi(t)=\int_t^\infty\frac{{\rm d}\theta}{\psi(\theta)}, \qquad t\in (\eta, \infty).
	\]
	The mapping $\phi:(\eta,\infty) \to (0,\infty)$ is bijective, and we write  $\varphi:(0,\infty)\to (\eta,\infty)$ for its right-continuous inverse. Thus
	\[
	\p_x(T_0<t)=\exp\{-x\varphi(t)\}.
	\]
	\item[iii)] if $\eta=0$ or if $\int_{0+}  {\rm d}\theta/ |\psi(\theta)|=\infty$, we have $\p_x(T_\infty<\infty)=0$,
	\item[iv)] if $\eta>0$ and  $\int_{0+}  {\rm d}\theta/ |\psi(\theta)|<\infty$, we define
	\[
	g(t)=-\int_0^t \frac{{\rm d}\theta}{\psi(\theta)}, \qquad t\in (0,\eta).
	\]
	The mapping $g:(0,\eta) \to (0,\infty)$ is bijective, we write  $\gamma:(0,\infty)\to (0,\eta)$ for its right-continuous inverse. Thus
	\[
	\p_x(T_{\infty}>t)=\exp\{-x\gamma(t)\}.
	\]
\end{itemize}
From (ii), we deduce that $\p_x(T_0<\infty)=\exp\{-x\eta \}$. Hence, the latter identity and (i) imply that a CB-process
has a
finite time extinction a.s. if and only if
\[
\eta<\infty, \qquad \int^{\infty}\frac{\ud u}{\psi(u)}<\infty\,\qquad \text{ and }\qquad \psi'(0+)\ge 0.
\]
Similarly from (iv), we get that $\p_x(T_{\infty}<\infty)=1-\exp\{-x\eta\}$. Hence from the latter and (iii), we deduce that a CB-process
has a
finite time explosion with positive probability if and only if
\[
\int_{0+}\frac{\ud u}{|\psi(u)|}<\infty\,\qquad \text{ and }\qquad \eta>0,
\]
When $\eta<\infty$, $\eta>0$ is equivalent to $\psi'(0+)< 0$.

The value of $ \psi'(0+)$ also determines  whether its associated CB-process will, on average, decrease, remain constant or increase. More precisely, under the assumption that  $q=0$, we observe that the first moment of a CB-process can be obtained by differentiating (\ref{CBut}) with respect to  $\lambda$. In particular, we may deduce   
\[
 \mathbb{E}_x[Y_t]=x e^{-\psi'(0^+)t}, \qquad \textrm{for}\quad t\ge 0.
 \]
Hence using  the same terminology as  for Bienaym\'e-Galton-Watson  processes, in respective order, a CB-process is called {\it supercritical, critical} or {\it subcritical} depending on the behaviour of  its mean, in other words on whether $\psi'(0^+)<0$, $\psi'(0^+)=0$ or $\psi'(0^+)>0$.  

The following two examples are of special interest in this paper since the Laplace exponent, i.e. the solution of (\ref{ut}),  can be computed explicitly in a closed form.  The first example that we present is the so-called Neveu branching process (see \cite{nev}) whose branching mechanism satisfies
$$\psi(\lambda)=\lambda\log( \lambda)=c\lambda+\int_{(0,\infty)}\big(e^{-\lambda x}-1+\lambda x{\mathbf 1}_{\{x<1\}}\big)\frac{\ud x}{x^{2}}, $$ where $c\in\R$ is a suitable constant. In this case $\psi^\prime(0+)=-\infty$,  $\eta =1$,  the process is supercritical and satisfies  the integral conditions of (i) and (iii). Thus, the Neveu branching process does not explode neither become extinct at a finite time a.s. According to Theorem 12.7 in \cite{Kyp}, we have that the Neveu branching process becomes extinct at infinity with positive probability. More precisely, if $(Y_t,t\ge 0)$ denotes the  Neveu branching process, then 
\[
\mathbb{P}_x\left(\lim_{t\to \infty} Y_t=0\right)=e^{-x}, \qquad x> 0.
\]
The second example is the stable case with drift,  in other words the branching mechanism satisfies 
\[
\psi(\lambda)=-a\lambda+c_\beta \lambda^{1+\beta}=-a \lambda+c_2 \lambda^{2}\mathbf{1}_{\{\beta=1\}}+c_\beta \frac{(\beta+1)\beta}{\Gamma(1-\beta)}\mathbf{1}_{\{\beta\neq 1\}}\int_{(0,\infty)}\big(e^{-\lambda x}-1+\lambda x\mathbf{1}_{\{\beta>0\}}\big)\frac{\ud x}{x^{2+\beta}}
\] with $\beta $ in $(-1,0)\cup (0,1]$ and $c_\beta$ is a non-zero constant with the same sign as $\beta$. It is known that its associated  CB-process can be obtained by scaling limits of Bienaym\'e-Galton-Watson  processes with a fixed   reproduction law. Moreover,  the case $\beta=1$ corresponds to the so-called  Feller diffusion branching process.

The  case  $\beta\in(-1,0)$ has a particular behaviour.  Its corresponding CB-process is  supercritical, since  $\psi^\prime(0+)=-\infty$. We observe that $\eta$ is infinite or  finite according to whether $a\ge 0$
or $a<0$.  The process satisfies the integral conditions of (i) and (iv). Thus  the stable CB process with $\beta\in (-1,0)$ does not  become extinct at a finite time a.s., and  it has a
finite time explosion with positive probability. Moreover,  the asymptotic behaviour of $\p_x(T_{\infty}>t)$ can be computed explicitly.

The  case  $\beta\in(0, 1]$  has a completely different behaviour.  Its associated CB-process is subcritical, critical or  supercritical depending on the value of $a$, since $\psi^\prime(0+)=-a$, and satisfies the integral conditions in (ii) and (iii). We also have that  $\eta=0$  or $\eta>0$ according as $a\le 0$ or $a>0$. In other words,  the stable CB-process with $\beta\in (0, 1]$ does not  explode at a finite time a.s., and  it becomes extinct at a finite time  with positive probability. Moreover the asymptotic behaviour of $\p_x(T_{0}<t)$ can be computed explicitly.

For our purposes, we recall that CB-processes can also be defined as the unique non-negative strong solution of the following stochastic differential equation (SDE for short)
\begin{equation}
\begin{split}\label{csbpsde}
Y_t=&Y_0+a\int_0^t Y_s \ud s+\int_0^t \sqrt{2\gamma^2 Y_s}\ud B_s \\
&\,\,\,+\int_0^t\int_{(0,1)}\int_0^{Y_{s-}}z\widetilde{N}(\ud s,\ud z,\ud u)+\int_0^t\int_{[1,\infty]}\int_0^{Y_{s-}}zN(\ud s,\ud z,\ud u),
\end{split}
\end{equation}
where $B=(B_t,t\geq 0)$ is a standard Brownian motion, $N(\ud s,\ud z,\ud u)$ is a Poisson random measure independent of $B$, with intensity $\ud s\Lambda(\ud z)\ud u$ where $\Lambda$ is a measure on $(0,\infty]$ defined as 
$\Lambda(\ud z)=\Ind{(0,\infty)}(z)\mu(\ud z)+q\delta_\infty(\ud z)$,   and $\widetilde{N}$ is the compensated measure of $N$, see for instance \cite{FuLi}.\\

 In this work, we are interested in studying  a particular class of CB-processes  in a random environment. More precisely, we are interested in the case when the random environment is driven by a Brownian motion which is independent of the dynamics of the original process. A process in this class is defined as the unique non-negative strong solution of a stochastic differential equation that conditioned on the environment,   satisfies the branching property.  We will refer to such class of processes as {\it CB-processes in a Brownian random environment}.

Our motivation  comes from the work of B\"oinghoff and Hutzenthaler \cite{CbMh} and Hutzenthaler \cite{Hu}, where they consider the case of  branching diffusions in a Brownian random environment  i.e.  when the branching mechanism  has no jump structure.  The authors in \cite{CbMh}  introduced this type of processes using a result of Kurtz  \cite{Kurtz}, where a  diffusion appro\-xi\-mation  of  Bienaym\'e-Galton-Watson in random environment is studied. The scaling limit obtained in \cite{Kurtz} turns out to be the strong solution of an SDE, that conditioned on the environment,  satisfies the branching property.  B\"oinghoff and Hutzenthaler computed the exact asymptotic behaviour of the survival probability using a time change method and in consequence, they also described the so called $Q$-process. This is  the process conditioned to be never extinct. Similarly to the discrete case, the authors in  \cite{CbMh} found a phase transition in the subcritical regime that depends on the parameters of the random environment. Hutzenthaler studied in \cite{Hu}, the supercritical regime and found that supercritical bran\-ching diffusions in a Brownian random environment conditioned on eventual extinction also possesses a phase transition which is similar to the phase transition of the subcritical case.

Another class of CB-processes in random environment has been studied recently by Bansaye et al. \cite{Bapa}. The authors in \cite{Bapa} studied the particular case where the  random environment is driven by a L\'evy process with paths of bounded variation.  Such type of processes  are called  CB-processes with catastrophes, motivated by the fact that the presence of a  negative jump in the random environment represents that a proportion of a population, following  the dynamics of the CB-process,  is killed.  Similarly to the diffusion case,  CB-processes with catastrophes were also introduced as the strong solution of an SDE and  conditioned on the environment,  they  satisfy the branching property. It is also important to note that CB-processes with catastrophes can also be obtained as the scaling limit of Bienaym\'e-Galton-Watson processes in random environment (see for instance \cite{basim}). Bansaye et al. also studied  the survival probability but unlike the case studied in \cite{CbMh}, they used a martingale technique since the time change  technique does not hold in general.  In the particular case where the branching mechanism is stable, the authors in \cite{Bapa} computed the exact asymptotic behaviour of the survival probability and obtained similar results to those found in \cite{CbMh}.

One of our aims is to study explosion and extinction probabilities for CB-processes in a Brownian random environment. Up to our knowledge, the explosion case has never been studied before even in the discrete setting. In order to  study explosion and extinction probabilities for a process in this class, we follow the  martingale technique used in \cite{Bapa} to compute the Laplace exponent   via a backward differential equation. With the Laplace exponent in hand, we are able to determine whether  a process is conservative, i.e. that does not explode a.s. at a fixed time, or become extinct  with positive probability. Nonetheless,  it seems very difficult to deduce necessary and sufficient conditions for explosion and extinction probabilities. This is due to the fluctuations of the random environment. However, the stable and the Neveu cases will help us to understand the different situations that our main results cannot cover.

   We give special attention to the case  when the branching mechanism is stable i.e. $\psi(\lambda)=-\alpha\lambda+c_\beta\lambda^{1+\beta}$, for $\beta\in (-1, 0)\cup (0, 1]$. Here, the Laplace exponent can be computed explicitly and we will show that it depends on the exponential functional of the random environment. Whenever  $\beta\in (-1, 0)$, we can compute the explosion probability at a fixed time 
and establish the asymptotic behaviour of the probability of no-extinction where we find three different regimes. Up to our knowledge, this behaviour was never observed. In the case when $\beta\in  (0, 1]$, we study the extinction probability and also establish the asymptotic behaviour of the survival probability where five different regimes appear. In the supercritical regime, we study the process conditioned on eventual extinction and we find three regimes for the asymptotic behaviour of the probability of  survival. The asymptotic behaviour depends on the study of exponential functionals of Brownian motion.
   
 From the speed of survival in the stable case, we can deduce the process conditioned to be never extinct or  $Q$-process using a Doob $h$-transform technique.
 
 We finish this paper studying the immigration case, which represents an example of  affine processes in a random environment. More precisely, this family of processes is an extension of the so-called Cox-Ingersoll-Ross model,  which is largely used in the financial literature, under the fluctuations of a random environment.

The remainder of the paper is structured as follows. In Section 2, we define and  study CB-processes in a Brownian random environment. Section 3 is devoted to  the long-term behaviour of this family of processes. In particular, we study  explosion and  extinction probabilities. In Section 4, we analyse the stable case. Here, we study the asymptotic behaviour of the no-explosion and survival probabilities as well as the process conditioned to be never extinct and the process conditioned on eventual extinction. Finally in Section 5, we study the immigration case.

\section{CB-processes in a Brownian random environment}

Motivated by the definition of branching diffusions in random environment (see B\"oinghoff and Hutzenthaler \cite{CbMh})  and CB-processes with catastrophes (see Bansaye et al. \cite{Bapa})  we introduce, using the same notation as in the SDE (\ref{csbpsde}),   continuous state branching processes in a Brownian random environment (in short a CBBRE) as the unique non-negative strong solution of the following stochastic differential equation
\begin{equation}
\begin{split}
\label{csbpbre}
 Z_t=&Z_0+\alpha \int_0^t Z_s \ud s+\int_0^t \sqrt{2\gamma^2 Z_s}\ud B_s +\sigma\int_0^t  Z_s \ud  B^{(e)}_s  \\
&+\int_0^t\int_{(0,1)}\int_0^{Z_{s-}}z\widetilde{N}(\ud s,\ud z,\ud u)+\int_0^t\int_{[1,\infty]}\int_0^{Z_{s-}}zN(\ud s,\ud z,\ud u),
\end{split}
\end{equation}
where $B^{(e)}=(B^{(e)}_t, t\geq 0)$ is a standard Brownian motion independent of $B$ and  the Poisson random measure $N$,  and $\alpha\in \mathbb{R}$. The Brownian motion  $B^{(e)}$ represents the {\it random environment}. We also observe that   the  drift coefficient  can be written as $\alpha=a+\alpha_0$,  where $a$ is the drift of the underlying CB-process and   $\alpha_0$ is the drift of the environment. 

The following theorem provides the existence of the CBBRE as a strong solution of (\ref{csbpbre}) and, in some sense, characterizes its law given the environment. In order to introduce our main result, we define  the  auxiliary process
\begin{align*}
K_t=\sigma B^{(e)}_t -\frac{\sigma^2}{2}t, \qquad \textrm{for }\quad t\geq 0,
\end{align*}
 that represents the {\it random environment}.
\begin{theorem}
The stochastic differential equation \eqref{csbpbre} has a unique non-negative strong solution. The process $Z=(Z_t, t\geq 0)$ satisfies the Markov property  and its infinitesimal generator $\mathcal{L}$ satisfies, for every $f\in C^2_b(\bar{\R}_+),$\footnote{${\R}_+=[0,\infty)$, $\bar{\R}_+=[0,\infty]$ and $C^2_b(\bar{\R}_+)$=\{twice differentiable functions such that $f(\infty)=0$\}}
\begin{equation}\begin{split}\label{bpgenerador}
\mathcal{L}f(x)=&\alpha xf'(x)+\left(\frac{1}{2}\sigma^2x^2+\gamma^2x\right)f''(x)\\
&\hspace{2cm}+x\int_{(0,\infty]} \left(f(x+z)-f(x)-zf'(x)\Ind{z<1}\right)\Lambda(\ud z).
\end{split}
\end{equation}
Furthermore, the process $Z$ conditioned on $K$, satisfies the branching property and for every $t>0$ and $z\geq 0$
\begin{align}\label{bplaplacek}
\mathbb{E}_z\Big[\exp\Big\{-\lambda Z_t e^{-K_t}\Big\}\Big|K\Big]=\exp\Big\{-zv_t(0,\lambda,K)\Big\}\qquad a.s.,
\end{align}
where for every $(\lambda,\delta)\in(\R_+,C(\R_+))$, $v_t: s\in[0,t]\mapsto v_t(s,\lambda,\delta)$ is the unique solution of the backward differential equation
\begin{align}\label{bpbackward}
\frac{\partial}{\partial s}v_t(s,\lambda, \delta)=e^{\delta_s}\psi(v_t(s,\lambda,\delta)e^{-\delta_s}), \,\qquad v_t(t,\lambda,\delta)=\lambda,
\end{align}
and $\psi$ is the branching mechanism of the underlying CB-process.
\end{theorem}

\begin{proof} First, we  prove the existence of a unique strong solution of the SDE
\begin{equation}\label{csbpbren}
\begin{split}
 Z_t^{(n)}&=Z_0^{(n)}+\alpha\int_0^t (Z_s^{(n)}\wedge n)\ud s+\int_0^t \sqrt{2\gamma^2 (Z_s^{(n)}\wedge n) }\ud B_s +\int_0^t  (Z_s^{(n)}\wedge n) \ud  B^{(e)}_s \\
&+\int_0^t\int_{(0,1)}\int_0^{(Z_{s}^{(n)}\wedge n)-}(z\wedge n) \widetilde{N}(\ud s,\ud z,\ud u)+\int_0^t\int_{[1,\infty]}\int_0^{(Z_{s}^{(n)}\wedge n)-}(z\wedge n) N(\ud s,\ud z,\ud u).
\end{split}
\end{equation}

Let $E=\{1,2\}$ and $\pi(\ud z)=\delta_1(\ud z)+\delta_2(\ud z)$ a measure in $E$. Define $W(\ud s,\ud z):=\ud B_s\delta_1(\ud z)+\ud B_s^{(e)} \delta_2(\ud z)$, a white noise on $[0,\infty)\times E$ with intensity $\ud s\pi(\ud z)$.  Then, the SDE (\ref{csbpbren}) can be written as follows
\begin{equation}\label{csbpbren1}
\begin{split}
Z_t^{(n)}=&Z_0^{(n)}+\alpha\int_0^t (Z_s^{(n)}\wedge n) \ud s\\
&+\int_0^t\int_E\left(\Ind{z=1}\sqrt{2\gamma^2 (Z_s^{(n)}\wedge n)}+\sigma\Ind{z=2}(Z_s^{(n)}\wedge n)\right) W(\ud s, \ud z) \\
& \hspace{-.5cm}+\int_0^t\int_{(0,1)}\int_0^{(Z_s^{(n)}\wedge n)-}(z\wedge n)\tilde{N}(\ud s,\ud z,\ud u)+\int_0^t\int_{[1,\infty)}\int_0^{(Z_s^{(n)}\wedge n)-}(z\wedge n)N(\ud s,\ud z,\ud u).
\end{split}
\end{equation}

\noindent Following the notation in \cite{DaLi}, the conditions from Theorem 2.5 in Dawson and Li \cite{DaLi} are satisfied by taking the spaces; $E=\{1,2\}$, $U_0=U_1=(0,\infty]\times(0,\infty)$,  the measures; $\pi(\ud z)=\delta_1(\ud z)+\delta_2(\ud z)$, $\mu_0(\ud z,\ud u)=\Ind{z<1}\mu(\ud z)\ud u$, $\mu_1(\ud z,\ud u)=\Ind{1\leq z<\infty}\mu(\ud z)\ud u +q\delta_\infty(\ud z)\ud u$, and the functions
\begin{align*}
b(x)& =\alpha(x\wedge n),& \sigma(x,z)=&\Ind{z=1}\sqrt{2\gamma^2 (x\wedge n)}+\sigma\Ind{z=2}(x\wedge n),\\
g_{0}(x,z,u)&=(z\wedge n)\Ind{u\leq x\wedge n},&  g_1(x,z,u)&=(z\wedge n)\Ind{u\leq x\wedge n}.
\end{align*}
Then, there exists a unique non-negative strong solution to (\ref{csbpbren1}) and therefore there exists a unique non-negative strong solution to (\ref{csbpbren}).  For $m\geq 1$ let $\tau_m=\inf\{t\geq 0: Z^{(m)}_t \geq m\}$. Since $0\leq Z^{(m)}_t\leq m$ for $0\leq t<\tau_m$, the trajectory $t\mapsto Z_t^{(m)}$ has no jumps larger than $m$ on the interval $[0,\tau_m)$. Then, we have that $Z_t^{(m)}$ satisfies (\ref{csbpbre}) for $0\leq t<\tau_m$. For $n\geq m\geq 1$, let $(Y_t, t\geq 0)$ be the strong solution to    
\begin{align*}
Y_t=&Z_{\tau_m-}^{(m)}+ \alpha\int_0^t (Y_s\wedge n) \ud s+\int_0^t \sqrt{2\gamma^2 (Y_s\wedge n) }\ud B_{\tau_m+s}  \\
&+\int_0^t  (Y_s\wedge n) \ud B^{(e)}_{\tau_m+s}+\int_0^t\int_{(0,1)}\int_0^{(Y_{s}\wedge n)-}(z\wedge n) \widetilde{N}(\tau_m+\ud s,\ud z,\ud u)\\
&\hspace{5cm}+\int_0^t\int_{[1,\infty]}\int_0^{(Y_{s}\wedge n)-}(z\wedge n) N(\tau_m+\ud s,\ud z,\ud u).
\end{align*}
Next, we define $Y_t^{(n)}=Z_t^{(m)}$ for  $0\leq t< \tau_m$ and $Y_t^{(n)}=Y_{t-\tau_m}$ for $t\geq \tau_m$. It is not difficult to see that $(Y_t^{(n)})$ is a solution to (\ref{csbpbren}). By the strong uniqueness, we get that $Z_t^{(n)}=Y_t^{(n)}$ for all $t\geq 0$. In particular, $Z_t^{(n)}=Z_t^{(m)}<m$, for $0\leq t <\tau_m$. Consequently, the sequence $\{\tau_m\}_{m\ge 1}$ is non-decreasing. We define the process $(Z_t,t\geq 0)$ as follows

\[ 
Z_t = \left\{ \begin{array}{ccl}
Z_t^{(m)} & \textrm{if} & t< \tau_m, \\
\infty & \textrm{if} & t\geq \underset{m\rightarrow\infty}{\lim}\tau_m. \end{array} \right.
\] 
Therefore, the  process $(Z_t, t\ge 0)$ is a weak solution to (\ref{csbpbre}). 

Now, we consider $Z'$ and $Z''$, two solutions  of  (\ref{csbpbre}) and define 
\[
\tau'_m=\inf\{t\geq 0: Z'_t \geq m\},\qquad \tau''_m=\inf\{t\geq 0: Z''_t \geq m\},
\] 
and $\sigma_m=\tau_m'\wedge \tau''_m$. Thus, $Z'$ and $Z''$ satisfy (\ref{csbpbren}) on $[0,\sigma_m)$, implying that both processes are indistinguishable on $[0,\sigma_m)$. If $\sigma_\infty=\underset{m\rightarrow\infty}{\lim}\sigma_m<\infty$, we have two situations to consider, either $Z'$ or $Z''$ explodes at $\sigma_\infty$ continuously or  by a jump of infinite size.  If $Z'$ or $Z''$ explodes continuously then both  processes explode at the same time since they  are  indistinguishable on $[0,\sigma_m)$, for all $m\ge 1$. If $Z'$ or $Z''$ explodes by a jump of infinite size, then this jump comes from an atom of the Poisson random measure $N$, so that both processes have it. Since after this time both processes are equal to $\infty$, and since the integral with respect to the Poisson process diverges, we get that $Z'$ and $Z''$ are indistinguishable. Then according to  Theorem 137 in Situ \cite{situ}, there is a unique strong solution to (\ref{csbpbre}) that we denote by $(Z_t, t\ge 0)$. The strong Markov property follows since  there is a strong solution, the integrators are L\'evy processes and the integrand functions are not time dependent (see for instance, Theorem V.32 in Protter \cite{Protter}, where the Lipschitz property  is just needed to guarantee the existence and uniqueness of the  strong solution) and  by It\^o's formula it is easy to show that the infinitesimal generator of $(Z_t, t\ge 0)$ is given by (\ref{bpgenerador}).

 The branching property of $Z_t$ conditioned on $K$, is inherited from the branching pro\-per\-ty of the underlying CB-process.  Its proof follows similar arguments as those used in Caballero et al. \cite{CLU} (see p.77-79).
 More precisely,  we consider $(Z_t, t\ge 0)$ and $(\widetilde{Z}_t, t\ge0)$  two solutions of (\ref{csbpbre})  starting from $x$ and $y$, respectively, which are independent conditioned on the  Brownian random environment $(B^{(e)}_t, t\ge 0)$. In other words, we consider $(B^{(1)}, N^{(1)})$ and $(B^{(2)}, N^{(2)})$  the Brownian motions and Poisson random measures associated to $Z$ and $\widetilde{Z}$, respectively, which are mutually independent.  
  Then conditioned on the environment, we observe
 \begin{align*}\label{branching}
 Z^\prime_t:=Z_t+\widetilde{Z}_t=x+y&+ \alpha\int_0^t Z^\prime_s \ud s+\int_0^t \sqrt{2\gamma^2 Z_s}\ud B_s^{(1)}  \nonumber \\
 &\hspace{2cm}+\int_0^t \sqrt{2\gamma^2 \widetilde{Z}_s}\ud B_s^{(2)}+\sigma\int_0^t  Z^\prime_s \ud B^{(e)}_s+ U_t+V_t,
  \end{align*}
 where
 $$U_t:=\int_0^t\int_{[1,\infty]}\int_0^{Z_{s-}}zN^{(1)}(\ud s,\ud z,\ud u)
 +\int_0^t\int_{[1,\infty]}\int_0^{\widetilde{Z}_{s-}}zN^{(2)}(\ud s,\ud z,\ud u)
 $$ 
 and $V_t$ represents the sum of the  integral terms with the compensated Poisson random measures.
 By the L\'evy's Characterization Theorem, we deduce
 $$W_t=\int_0^t\Ind{Z^\prime_s\neq 0}\frac{\sqrt{2\gamma^2 Z_s}\ud B_s^{(1)}+\sqrt{2\gamma^2 \widetilde{Z}_s}\ud B_s^{(2)}}{\sqrt{2\gamma^2 Z^\prime_s}}  +\int_0^t\Ind{Z^\prime_s= 0}\ud B_s^{(1)}$$
 is a Brownian motion, since $\langle W\rangle_t=t$, for $t\ge 0$. 
 
 Next, we introduce $M(\ud s,\ud z,\ud u)$  an independent Poisson random measure, with intensity $\ud s\Lambda(\ud z)\ud u$ and define
\[
\begin{split}
N(\ud s,\ud z,\ud u)&=\Ind{u< Z_{s-}}N^{(1)}(\ud s,\ud z,\ud u)\\
&+\Ind{Z_{s-}<u<Z^\prime_t-}N^{(2)}(\ud s,\ud z,\ud u)+\Ind{Z^\prime_{s-}<u}M(\ud s,\ud z,\ud u).
\end{split}
\]
Using the same arguments as in \cite{CLU} (p. 78-79), one can deduce   that $N$ is a Poisson random measure with intensity $\ud s\Lambda(\ud z)\ud u$, and 
$$U_t=\int_0^t\int_{[1,\infty]}\int_0^{Z^\prime_{s-}}z N(\ud s,\ud z,\ud u) \qquad V_t=\int_0^t\int_{(0,1)}\int_0^{Z^\prime_{s-}}z\widetilde{N}(\ud s,\ud z,\ud u).$$
Therefore, given the environment,  $Z^\prime$ satisfies the SDE (\ref{csbpbre}) with $Z^\prime_0=x+y$, implying that it satisfies the branching property.

  In order to deduce (\ref{bplaplacek}), we follow similar arguments as  used in Bansaye \cite{Bapa}. To this purpose, we introduce $\widetilde{Z}_{t}=Z_{t}e^{-K_{t}}$ and take $F\in C^{1,2}(\R_+, \bar{\R}_+)$. An application of  It\^o's formula   guarantees that $F(t,\widetilde{Z}_{t})$ conditioned on $K$ is a local martingale if and only if for every $t\geq 0$,
\[
\begin{split}
0&=\int_{0}^t \left(\frac{\partial}{\partial t}F(s, \widetilde{Z}_s)+\alpha\widetilde{Z}_s\frac{\partial}{\partial x}F(s, \widetilde{Z}_s)+\gamma^2 e^{-K_s}\widetilde{Z}_s\frac{\partial^2}{\partial x^2}F(s, \widetilde{Z}_s) \right)\ud s \\
&+\int_0^t\int_0^{\infty} Z_s\left( F(s, \widetilde{Z}_s+ze^{-K_s})-F(s, \widetilde{Z}_s)-\frac{\partial}{\partial x}F(s, \widetilde{Z}_s)ze^{-K_s}\Ind{z<1}\right)\Lambda(\ud z)\ud s.
\end{split}
\]
If in addition, $F$ is bounded, it will be a true martingale if the previous equality holds. By choosing $F(s,x)=\Expo{-xv_t(s,\lambda, K)}$, where $v_t(s,\lambda,K)$ is differentiable with respect to the variable $s$, non-negative and such that  $v_t(t,\lambda, K)=\lambda$ for all $\lambda\geq 0$, we observe that $(F(s,\widetilde{Z}_s),s\leq t)$ conditioned on $K$ is a martingale if and only if 
\begin{align*}
\frac{\partial}{\partial s}v_t(s,\lambda,K)&=-qe^{K_s}-av_t(s,\lambda,K)+\gamma^2(v_t(s,\lambda,K))^2e^{-K_s}\\
&+ e^{K_s}\int_{0}^{\infty}\left(e^{-e^{-K_s}v_t(s,\lambda, K)z}-1+e^{-K_s}v_t(s,\lambda, K)z\Ind{z<1}\right)\mu(\ud z),
\end{align*}
which is equivalent to $v_t(s,\lambda, K)$ satisfying (\ref{bpbackward}). Since  $\psi$ is locally Lipschitz on $(0,\infty)$, the existence and uniqueness of $v_t$ follows from the Picard-Lindel\"of theorem. 

The above implies that  the process $(\exp\{-\widetilde{Z}_{s}v_{t}(s,\lambda, K)\}, 0\leq s\leq t)$ conditioned on $K$ is a martingale, and hence 
$$\Expconx{z}{\Expo{-\lambda \widetilde{Z_t}}}{K}=\Expconx{z}{\Expo{- \widetilde{Z}_0v_t(0,\lambda,K)}}{K}=\Expo{-zv_t(0,\lambda,K)},$$
which completes the proof.
\end{proof}
\begin{rem} Observe that in the case when $|\psi^\prime(0+)|<\infty$,  the  auxiliary process can be taken as follows
\[
K^{(0)}_t=\sigma B^{(e)}_t+ \mathbf{m}t, \qquad \textrm{for }\quad t\geq 0,
\]
where
\[
\mathbf{m}=-\psi^\prime(0+)-\frac{\sigma^2}{2}.
\]
Following  the same arguments as in the last part of the proof of the previous Theorem and  replacing $K$ with $K^{(0)}$, one can deduce that $v_t(s,\lambda, K^{(0)})$ is the unique solution to  the backward differential equation
\begin{equation}\label{edbmf}
\frac{\partial}{\partial s}v_t(s,\lambda, K^{(0)})=e^{K^{(0)}_s}\psi_0(v_t(s,\lambda,K^{(0)})e^{-K^{(0)}_s}),
\end{equation}
where
\begin{equation}\label{psi0}
\psi_0(\lambda)=-q+\gamma^2
\lambda^2+\int_{(0,\infty)}\big(e^{-\lambda x}-1+\lambda x\big)\mu(\ud x).
\end{equation}
In this case, the process $Z$ conditioned on $K^{(0)}$, satisfies, for every $t>0$ and $z\geq 0$,
\begin{align}\label{bplaplacek1}
\mathbb{E}_z\Big[\exp\Big\{-\lambda Z_t e^{-K^{(0)}_t}\Big\}\Big|K^{(0)}\Big]=\exp\Big\{-zv_t(0,\lambda,K^{(0)})\Big\}\qquad a.s.
\end{align}

\end{rem}

Before we continue with the exposition of this manuscript, we would like to provide some examples where we can compute explicitly the Laplace exponent of the CB-process in a Brownian random environment. \\

 {\bf Example 1.  (Neveu case):} The  Neveu branching process in a Brownian random environment has branching mechanism given by $\psi(u)=u\log(u)$, for $u\ge 0$. In this particular case the backward differential  equation (\ref{bpbackward}) satisfies
$$\frac{\partial}{\partial s}v_t(s,\lambda, \delta)=v_t(s,\lambda, \delta)\log(e^{-\delta_s} v_t(s,\lambda, \delta)).$$
Providing that $v_t(t,\lambda, \delta)=\lambda$, one can solve the above equation and after some straightforward computation we deduce 
$$v_t(s, \lambda, \delta)=\Expo{e^{s}\left(\int_s^t e^{-u}\delta_u \ud u+\log(\lambda)e^{-t}\right)}.$$
Hence, from identity (\ref{bplaplacek}) we get
\begin{align}\label{bplaplacekneveu}
\e_z\Big[\exp\Big\{-\lambda Z_t e^{-K_t}\Big\}\Big| K\Big]=\Expo{-z\lambda^{e^{-t}}\Expo{\int_0^t e^{-s}K_s \ud s}}\qquad \textrm{a.s.}
\end{align}
Observe that the r.v. $\int_0^t e^{-s}K_s \ud s$,  is normal distributed   with mean $- \frac{\sigma^2}{2}(1-e^{-t}-te^{-t})$ and variance $\frac{\sigma^2}{2}(1+4e^{-t}-3e^{-2t}),$ for $t\ge 0.$ In other words, the Laplace transform of $Z_t e^{-K_t}$ can be determined by the Laplace transform of a log-normal distribution which we know  exists but there is not an explicit form of it. \\

 {\bf Example  2. (Feller case):} Assume that  $\mu(0,\infty)=0$, thus the CB-process in a Brownian random environment (\ref{csbpbre}) is reduced to  the following SDE
$$Z_t=Z_0+\alpha\int_0^t  Z_s\ud s+\sigma\int_0^t  Z_s\ud B^{(e)}_s+ \int_0^t \sqrt{2\gamma^2 Z_s}\ud B_s.$$
This SDE is equivalent to the strong solution of the SDE
\begin{align*}
\ud Z_t=&\frac{\sigma^2}{2}Z_t \ud t+Z_t\ud K_t+\sqrt{2\gamma^2 Z_s}\ud B_s,\\
\ud K_t=&a_0\ud t+\sigma \ud B^{(e)}_t,
\end{align*}
where $a_0=\alpha -\sigma^2/2$,  which is the branching diffusion in random environment studied by B\"oinghoff and Hutzenthaler \cite{CbMh}.

The backward differential  equation (\ref{bpbackward}) satisfies
$$\frac{\partial}{\partial s}v_t(s,\lambda, \delta)=-\alpha v_t(s,\lambda, \delta)+\gamma^2 v^2_t(s,\lambda, \delta) e^{-\delta_s}.$$
If $v_t(t,\lambda, \delta)=\lambda$,  the above equation can be solved and after some computations one can deduce
$$v_t(s, \lambda, \delta)=e^{-\alpha s}\left((\lambda e^{\alpha t})^{-1}+ \gamma^2\int_s^t e^{-(\delta_u+\alpha u)} \ud u\right)^{-1}.$$
Hence, from identity (\ref{bplaplacek}) we get
\begin{align}\label{bplaplace2}
\e_z\Big[\exp\Big\{-\lambda Z_t e^{-(K_t+\alpha t)}\Big\}\Big| K\Big]=\Expo{-z\left(\lambda^{-1}+\gamma^2\int_0^t e^{-(K_u+\alpha u)} \ud u\right)^{-1}}\qquad \textrm{a.s.}
\end{align}
The r.v. $\int_0^t e^{-(K_u+\alpha u)} \ud u,$ is known as the exponential functional a Brownian motion with drift, and it has been deeply studied by many authors, see for instance \cite{beryor, Du, MaYor}. \\

{\bf Example  3. (Stable case): }  For our last example, we assume that  the  branching mechanism is of the form
$$\psi(\lambda)=-\alpha\lambda +c_\beta\lambda^{\beta+1}, \qquad \lambda \geq 0,$$
for some $\beta\in(-1,0)\cup(0,1)$, $\alpha\in \mathbb{R}$, and $c_\beta$ is such that
\[
\left\{ \begin{array}{ll}
 c_\beta<0 &\textrm{ if $\beta \in(-1,0),$}\\
c_\beta>0 & \textrm{ if $\beta \in (0,1).$}
 \end{array} \right .
\]
Under this assumption, the process $Z$ satisfies the following stochastic differential equation 
\begin{align}\label{csbpbrest}
Z_t =& Z_0 +\alpha\int^t_0 Z_s\ud s +\sigma\int_0^t  Z_s\ud B^{(e)}_s+\int_0^t\int_0^{\infty}\int_0^{Z_{s-}}z\widehat{N}(\ud s,\ud z,\ud u) 
\end{align}
where $ B^{(e)}$ is a standard Brownian motion   and
 \[
 \widehat{N}(\ud s,\ud z,\ud u) =\left\{ \begin{array}{ll}
 N(\ud s,\ud z,\ud u)  &\textrm{ if $\beta \in(-1,0),$}\\
 \widetilde{N} (\ud s,\ud z,\ud u) & \textrm{ if $\beta \in (0,1),$}
 \end{array} \right .
 \]
 where $N$ is an independent Poisson random measure with intensity 
\[
\frac{c_\beta\beta(\beta+1)}{\Gamma(1-\beta)}\frac{1}{z^{2+\beta}}\ud s \ud z\ud u,
\]
and $\widetilde{N}$ is its  compensated version.

In this  case, we note 
 \[
 \psi^\prime(0+) =\left\{ \begin{array}{ll}
-\infty  &\textrm{ if $\beta \in(-1,0),$}\\
-\alpha & \textrm{ if $\beta \in (0,1).$}
 \end{array} \right .
 \]
Hence, when $\beta \in (0,1)$, we have  $K^{(0)}_t=K_t+\alpha t$, for $t\ge 0$. In both cases, we use the backward differential  equation (\ref{bpbackward}) and observe that it satisfies
$$\frac{\partial}{\partial s}v_t(s,\lambda, \delta)=-\alpha v_t(s,\lambda, \delta)+c_\beta  v^{\beta+1}_t(s,\lambda, \delta) e^{-\beta \delta_s}.$$
Similarly to the Feller case, assuming that  $v_t(t,\lambda, \delta)=\lambda$, we can solve the above equation  and after some straightforward computations, we get
$$v_t(s, \lambda, \delta)=e^{-\alpha s}\left((\lambda e^{\alpha t})^{-\beta}+\beta c_\beta\int_s^t e^{-\beta(\delta_u+\alpha u)} \ud u\right)^{-1/\beta}.$$
Hence, from (\ref{bplaplacek}) we get the following a.s. identity
\[
\e_z\Big[\exp\Big\{-\lambda Z_t e^{-K_t}\Big\}\Big| K\Big]=\Expo{-z\left((\lambda e^{\alpha t})^{-\beta}+\beta c_\beta\int_0^t e^{-\beta(K_u+\alpha u)} \ud u\right)^{-1/\beta}},
\]
which clearly implies the following a.s. identity
\begin{align}\label{bplaplace3}
\e_z\Big[\exp\Big\{-\lambda Z_t e^{-(K_t+\alpha t)}\Big\}\Big| K\Big]=\Expo{-z\left(\lambda^{-\beta}+\beta c_\beta\int_0^t e^{-\beta(K_u+\alpha u)} \ud u\right)^{-1/\beta}}.
\end{align}
We finish this example by observing that the  r.v. $\int_0^t e^{-\beta (K_u+\alpha u)} \ud u$ is the exponential functional of the Brownian motion with drift $(\beta (K_u+\alpha u), u\ge 0)$. 
\section{Long-term behaviour}
Similarly to the CB-processes case,  there are three events which are of immediate concern for the process $Z$, {\it explosion},  {\it absorption} and {\it extinction}. Recall that the event of explosion at fixed time $t$, is given by  $\{Z_t=\infty\}$. When $\mathbb{P}_z(Z_t<\infty)=1$, for all $t>0$ and $z>0$, we say the process is {\it conservative}. In the second event, we observe from the definition of $Z$ that if $Z_t=0$ for some  $t>0$, then $Z_{t+s}=0$ for all $s\ge 0$, which makes $0$ an absorbing state. As $Z_t$ is to be thought of as the size of a given population at time $t$, the event $\{\lim_{t\to \infty}Z_t=0 \}$ is referred as {\it extinction}.

Up to our knowledge, {\it explosion} has never been studied before for branching processes in random environment even in the discrete setting. Most of the results that appear in the literature are related to {\it extinction}. In this section, we first provide a sufficient condition under which the process $Z$ is conservative and an example where we can determine explicitly the probability of explosion. Under the condition that the process is conservative, we study the probability of extinction under the  influence of the random environment.

In our particular case,  the events of explosion and absorption are not so easy to deduce in full generality. In the next section, we provide an example under which both events can be computed explicitly, as well as their asymptotic behaviour when time increases.

\subsection{Explosion and conservative processes}

Recall that $\psi^\prime(0+)\in [ -\infty, \infty)$, and that whenever $|\psi^\prime(0+)|<\infty$, we write
\[
\mathbf{m}=-\psi^\prime(0+)-\frac{\sigma^2}{2}.
\]
The following proposition provides necessary conditions under which the process $Z$ is conservative.
\begin{proposition} Assume that $q=0$ and  $|\psi'(0+)|<\infty$,  then a CBBRE  with  branching mechanism $\psi$ is conservative.
\end{proposition}

\begin{proof} Recall that under our assumption, the auxiliary process  takes the form
\[
K^{(0)}_t=\sigma B^{(e)}_t+ \mathbf{m}t, \qquad \textrm{for }\quad t\geq 0,
\]
and that  $v_t(s,\lambda, K^{(0)})$ is the unique solution to  the backward differential equation (\ref{edbmf}) with $q=0$.
From identity (\ref{bplaplacek1}), we know that
\[
\mathbb{E}_z\Big[\exp\Big\{-\lambda Z_t e^{-K^{(0)}_t}\Big\}\Big]=\mathbb{E}\Big[\exp\Big\{-zv_t(0,\lambda,K^{(0)})\Big\}\Big].
\]
Thus if we take limits as $\lambda\downarrow 0$, we deduce that for $z>0$
\[
\mathbb{P}_z\big(Z_t<\infty\big)=\underset{\lambda\downarrow 0}{\lim}\,\mathbb{E}_z\Big[\exp\Big\{-\lambda Z_t e^{-K^{(0)}_t}\Big\}\Big]=\Exp{\Expo{-z\underset{\lambda\downarrow 0}{\lim}\ v_t(0,\lambda,K^{(0)})}},
\]
where the limits are justified by Monotonicity and Dominated Convergence. This implies that  a CBBRE  is conservative if and only if 
\[
\lim_{\lambda \downarrow 0}v_t(0,\lambda,K^{(0)})=0.
\]
Let us introduce the function $\Phi(\lambda)=\lambda^{-1}\psi_0(\lambda)$, where $\psi_0$ is given by \eqref{psi0} and  observe that  $\Phi(0)=\psi'_0(0+)=0$. Since $\psi_0$ is convex, we deduce that $\Phi$ is increasing. Finally, if we solve equation (\ref{edbmf}) with $\psi_0(\lambda)=\lambda \Phi(\lambda)$, we get 
$$v_t(s,\lambda,K)=\lambda \Expo{-\int_s^t \Phi(e^{-K_r}v_t(r,\lambda,K))\ud r}.$$
Therefore, since $\Phi$ is increasing and $\Phi(0)=0$, we have
$$0\leq\underset{\lambda\rightarrow 0}{\lim}v_t(0,\lambda,K)=\underset{\lambda\rightarrow 0}{\lim}\lambda\Expo{-\int_0^t \Phi(e^{-K_r}v_t(r,\lambda,K))\ud r}\leq \underset{\lambda\rightarrow 0}{\lim}\lambda=0,$$
implying that $Z$ is conservative.

\end{proof}
Recall that in  the case when there is no random environment, i.e. $\sigma=0$,  we know that a CB-process with  branching mechanism $\psi$ is conservative if and only if
$$\int_{0+} \frac{\ud u}{|\psi(u)|}=\infty.$$

In the case when the random environment is present, it is not so clear how to get a  necessary and sufficient condition in terms of the branching mechanism since the random environment is affecting the monotonicity of $\psi$ in the backward differential equation (\ref{bpbackward}).

We now provide  two interesting examples in the case when $\psi^\prime(0+)=-\infty$, that behave completely differently.\\

{\bf1. Stable case with $\beta\in(-1,0)$.} Recall that in this case $\psi(u)=-\alpha u +c_\beta u^{\beta+1}$, where $\alpha\in \mathbb{R}$ and $c_\beta$ is a negative constant. From straightforward computations, we get
\[
 \psi^\prime(0+)=-\infty,\qquad \mathrm{and}\qquad \int_{0+}\frac{\ud u}{|\psi(u)|}<\infty,
\]
On the other hand, from identity (\ref{bplaplace3}) and taking limits as $\lambda\downarrow 0$, we deduce that for $z>0$
\begin{align}\label{eqexplo}
\mathbb{P}_z\Big( Z_t < \infty\Big|K\Big)=\Expo{-z\left(\beta c_\beta\int_0^t e^{-\beta(K_u+\alpha u)} \ud u\right)^{-1/\beta}}\qquad \textrm{a.s.,}
\end{align}
implying 
\[
\mathbb{P}_z\Big( Z_t =\infty\Big | K\Big)=1-\Expo{-z\left(\beta c_\beta\int_0^t e^{-\beta(K_u+\alpha u)} \ud u\right)^{-1/\beta}}>0.
\]
In other words the stable CBBRE  with $\beta \in (-1, 0)$ explodes with positive probability for any $t>0$. Moreover, if  the process
$(K_u+\alpha u, u\ge 0)$ does not drift to $-\infty$, i.e. $\alpha-\sigma^2/2\ge 0$, 
we deduce from Theorem 1 in Bertoin and Yor \cite{beryor}, 
\begin{equation*}
\int_0^t e^{-\beta(K_u+\alpha u)} \ud u\to\infty, \qquad \textrm{as} \quad t\to \infty,
\end{equation*}
implying 
 \[
\lim_{t\to \infty} Z_t=\infty,\qquad \textrm{a.s.}
\]
On the other hand, if the process $(K_u+\alpha u, u\ge 0)$ drifts to $-\infty$, i.e. $\alpha-\sigma^2/2< 0$, we have an interesting long-term behaviour of the process $Z$. In fact, we deduce from the Dominated Convergence Theorem
\[
\mathbb{P}_z\Big( Z_\infty =\infty\Big)=1-\mathbb{E}\left[\Expo{-z\left(\beta c_\beta\int_0^\infty e^{-\beta(K_u+\alpha u)} \ud u\right)^{-1/\beta}}\right].
\]
The above probability is positive since 
\[
\int_0^\infty e^{-\beta(K_u+\alpha u)} \ud u<\infty\qquad\textrm{a.s.},
\]
according to Theorem 1 in Bertoin and Yor \cite{beryor}. In  this particular  case, we will discuss the  asymptotic behaviour of the probability of explosion  in  Section 4.\\

{\bf 2. Neveu case.} In this case, recall that  $\psi(u)=u\log u$. In particular
\[
 \psi^\prime(0+)=-\infty\qquad \textrm{and}\qquad \int_{0+}\frac{\ud u}{|u\log u|}=\infty.
\]
By taking limits as $\lambda\downarrow 0$ in (\ref{bplaplacekneveu}),  one can see that the process is conservative  conditionally on  the environment, i.e. 
\begin{align*}
\Proconx{z}{Z_t<\infty}{K}&=1,
\end{align*}
for all $t\in (0,\infty)$ and $z\in[0,\infty)$.

\subsection{Extinction probabilities}
Here, we consider CBBREs  that are conservative. The following result provides a criteria that depends on the behaviour of the auxiliary process $K^{(0)}$, which allows us to compute the probability of extinction of a CBBRE. Recall that the event of extinction is equal to $\{\lim_{t\to \infty}Z_t=0\}$.
\begin{proposition} Assume that $q=0$ and $|\psi^\prime(0+)|<\infty$.  
Let $(Z_t, t\ge 0)$ be a CBBRE with  branching mechanism given by $\psi$ and $z>0$.
\begin{enumerate}
\item[i)] If $\mathbf{m}<0$, then $\mathbb{P}_z\Big(\underset{t\rightarrow\infty}{\lim} Z_t=0\Big|K^{(0)}\Big)=1$, a.s.
\item[ii)] If $\mathbf{m}=0$, then $\mathbb{P}_z\left(\underset{t\rightarrow\infty}{\liminf} Z_t=0\Big| K^{(0)}\right)=1$, a.s. Moreover if $\gamma>0$ then 
\[
\mathbb{P}_z\Big(\underset{t\rightarrow\infty}{\lim}Z_t=0\Big|K^{(0)}\Big)=1,  \textrm{a.s.}
\]
\item[iii)] If $\mathbf{m}>0$ and 
\[
 \int^\infty x\ln (x)\, \mu({\rm d} x)<\infty,
\]  
then $\mathbb{P}_z\Big(\underset{t\rightarrow\infty}{\liminf}Z_t>0\Big|K^{(0)}\Big)>0 $ a.s.,  and there exists a non-negative finite r.v. $W$ such that
$$Z_te^{-K^{(0)}_t}\underset{t\rightarrow \infty}{\longrightarrow}W,\ \textrm{a.s} \qquad \textrm{and}\qquad \big\{W=0\big\}=\Big\{\lim_{t\rightarrow\infty}Z_t=0\Big\}.$$
Moreover if $\gamma>0$, we have
\[
\mathbb{P}_z\left(\lim_{t\to \infty}Z_t=0\Big| K^{(0)}\right)>0, \quad \textrm{a.s.}
\]
and, 
\[
 \mathbb{P}_z\left(\lim_{t\to \infty}Z_t=0\right)\ge \left(1+\frac{z\sigma^2}{\gamma^2}\right)^{-\frac{2\mathbf{m}}{\sigma^{2}}}.
\]
\end{enumerate}
\end{proposition}

\begin{proof} Recall that under our assumption, i.e. $|\psi^\prime(0+)|<\infty$,   the auxiliary process can be written as 
\[
K^{(0)}_t=\sigma B^{(e)}_t+ \mathbf{m}t, \qquad \textrm{for }\quad t\geq 0,
\]
and   the function $v_t(s,\lambda, K^{(0)})$ satisfies the backward differential equation (\ref{edbmf}).

Similarly to  the last part of the proof of Theorem 1, one can prove that $Z_te^{-K^{(0)}_t}$ is a non-negative local martingale. Therefore $Z_te^{-K^{(0)}_t}$ is a non-negative supermartingale and it converges a.s. to a non-negative finite random variable, here denoted by  $W$.  This implies the statement of part (i) and the first statement of  part  (ii). 

 In order to prove the second statement of part (ii), we  observe that  if   $\gamma>0$, then the backward differential equation (\ref{edbmf}) satisfies
\begin{align*}
\frac{\partial}{\partial s}{v}_t(s,\lambda, K^{(0)})\ge \gamma^2{v}_t(s,\lambda,K^{(0)})^2 e^{-K^{(0)}_s}.
\end{align*}
Therefore
\[
v_t(s, \lambda, K^{(0)}) \le \left(\frac{1}{\lambda}+\gamma^2\int_s^te^{-K^{(0)}_s}{\rm d}s\right)^{-1}, 
\]
which implies the following inequality,
\begin{equation}\label{desgamma}
\mathbb{P}_{z}(Z_t=0| K^{(0)})\ge \exp\left\{-z\left(\gamma^2\int_0^te^{-K^{(0)}_s}{\rm d}s\right)^{-1}\right\}.
\end{equation}
Since  $\mathbf{m}\le0$, we deduce from Theorem 1 in Bertoin and Yor \cite{beryor} 
\begin{equation}\label{BerYor1}
\int_0^te^{-K^{(0)}_s}{\rm d}s\to\infty, \qquad \textrm{as} \quad t\to \infty.
\end{equation}
In consequence, we have
\[
\mathbb{P}_z\left(\lim_{t\to \infty}Z_t=0\Big|K^{(0)} \right)=1 \quad \textrm{a.s.}
\]

  Now, we prove part (iii). We first note that  ${v}_t(\cdot,\lambda, K^{(0)})$, the solution to the backward differential equation (\ref{edbmf}), is non-decreasing on $[0,t]$ (since $\psi_0$ is positive). Thus for all $s\in[0,t]$, ${v}_t(s,\lambda,K^{(0)})\leq \lambda$.  From the proof of part (ii) in Proposition 1, we know that the function $\Phi(\lambda)=\lambda^{-1}\psi_0(\lambda)$,  is increasing.   Hence
\begin{align*}
\frac{\partial}{\partial s}{v}_t(s,\lambda, K^{(0)})&={v}_t(s,\lambda,K^{(0)})\Phi({v}_t(s,\lambda, K^{(0)})e^{-K^{(0)}_s} )\leq {v}_t(s,\lambda,K^{(0)})\Phi(\lambda e^{-K^{(0)}_s}).
\end{align*}
Therefore, for every $s\le t$, we have
\[
v_t(s, \lambda, K^{(0)})\ge \lambda \exp\left\{-\int_{s}^t\Phi(\lambda e^{-K^{(0)}_s}){\rm d} s \right\}.
\]
In particular, 
\[
\liminf_{t\to\infty} v_t(0, \lambda, K^{(0)})\ge \lambda \exp\left\{-\int_{0}^\infty\Phi(\lambda e^{-K^{(0)}_s}){\rm d} s \right\}.
\]
According to Proposition 3.3  in Salminen and Yor \cite{SY}, we have
\[
\int_{0}^\infty\Phi(\lambda e^{-K^{(0)}_s}){\rm d} s <\infty \quad \textrm{a.s.,} \qquad \textrm{if and only if}\qquad \int_{0}^\infty\Phi(\lambda e^{-y}){\rm d} y <\infty.
\]
Observe 
\[
\begin{split}
\int_{0}^\infty\Phi(\lambda e^{-y}){\rm d} y &=\int_{0}^\lambda\frac{\Phi(\theta)}{\theta}{\rm d} \theta\\
&=\gamma^2\lambda +\int_0^\lambda \frac{{\rm d} \theta}{\theta^2}\int_{(0,\infty)} (e^{-\theta x}-1+\theta x)\mu ({\rm d}x)\\
&=\gamma^2\lambda +\int_{(0,\infty)}\mu ({\rm d}x)\int_0^\lambda (e^{-\theta x}-1+\theta x)\frac{{\rm d} \theta}{\theta^2} \\
&=\gamma^2\lambda +\int_{(0,\infty)} x\left(\int_0^{\lambda  x}(e^{-y}-1+y)\frac{{\rm d} y}{y^2} \right )\mu ({\rm d}x).
\end{split}
\]
Since the function
\[
g_\lambda(x)=\int_0^{\lambda  x}(e^{-y}-1+y)\frac{{\rm d} y}{y^2},
\]
is equivalent to $\lambda x/2$ as $x\to 0$ and equivalent to  $\ln x$ as $x\to \infty$, we deduce that 
\[
\int_{0}^\infty\Phi(\lambda e^{-y}){\rm d} y<\infty \qquad  \textrm{if and only if } \qquad \int^\infty x\ln x \mu({\rm d}x )<\infty.
\]
In other words, 
\[
\int_{0}^\infty\Phi(\lambda e^{-K^{(0)}_s}){\rm d} s <\infty \quad \textrm{a.s.,} \qquad \textrm{if and only if}\qquad \qquad \int^\infty x\ln x \mu({\rm d}x )<\infty.
\]
If the integral condition from above is satisfied, then
\[
\liminf_{t\to\infty} v_t(0, \lambda, K^{(0)})\ge \lambda \exp\left\{-\int_{0}^\infty\Phi(\lambda e^{-K^{(0)}_s}){\rm d} s \right\}>0, 
\]
implying 
 \[
 \mathbb{E}_z\Big[e^{-\lambda W}\Big| K^{(0)}\Big]\le \exp\left\{-z\  \lambda \exp\left\{-\int_{0}^\infty\Phi(\lambda e^{-K^{(0)}_s}){\rm d} s \right\}\right\}<1,
 \]
 and in particular $\mathbb{P}_z\Big(\underset{t\rightarrow\infty}{\liminf}Z_t>0\Big|K^{(0)}\Big)>0 $ a.s. Next, we use Lemma 20 in \cite{Bapa} and the branching property of $Z$, to deduce
$$\{W=0\}=\Big\{\underset{t\rightarrow\infty}{\lim}Z_t=0\Big\}.$$
Now assume  $\gamma>0$. From inequality (\ref{desgamma}) and the Dominate Convergence Theorem, we deduce
\[
\mathbb{P}_z\left(\lim_{t\to \infty}Z_t=0\Big| K^{(0)}\right)\ge \exp\left\{-z\left(\gamma^2\int_0^\infty e^{-K^{(0)}_s}{\rm d}s\right)^{-1}\right\} \quad \textrm{a.s.}
\]
On the other hand, since   $\mathbf{m}>0$, we deduce from Theorem 1 in Bertoin and Yor \cite{beryor} (or Proposition 3.3  in Salminen and Yor \cite{SY})
\begin{equation}\label{BerYor2}
\int_0^\infty e^{-K^{(0)}_s}{\rm d}s<\infty, \qquad \textrm{a.s.},
\end{equation}
implying that $\mathbb{P}_z\big(\lim_{t\to \infty}Z_t=0\big| K^{(0)}\big)>0$, a.s. Finally, according to Dufresne \cite{Du}, 
\begin{equation}\label{Dufresne}
\int_0^\infty e^{-K^{(0)}_s}{\rm d}s\quad \textrm{ has the same law as }\quad \Big(2\Gamma_{\frac{2\mathbf{m}}{\sigma^2}}\Big)^{-1},
\end{equation}
where $\Gamma_v$ is Gamma r.v. with parameter $v$. After straightforward computations, we deduce
\[
 \mathbb{P}_z\left(\lim_{t\to \infty}Z_t=0\right)\ge \left(1+\frac{z\sigma^2}{\gamma^2}\right)^{-\frac{2\mathbf{m}}{\sigma^{2}}}.
\]
The proof of  Proposition 2 is now complete. 
\end{proof}
\begin{rem} When $\mathbf{m}>0$, we also  have an upper bound for the probability of extinction under the assumption 
\[
 \gamma>0\qquad  \textrm{and}\qquad \kappa:=\int_1^\infty x^2\mu({\rm d} x)<\infty.
 \] 
 More precisely, under the above assumption we have, for $z>0$
\[
\left(1+\frac{z\sigma^2}{2\gamma^2}\right)^{-\frac{2\mathbf{m}}{\sigma^{2}}}\leq \mathbb{P}_z\left(\lim_{t\to \infty}Z_t=0\right)\leq \left(1+\frac{1}{2}\frac{z\sigma^2}{\gamma^2+\kappa}\right)^{-\frac{2\mathbf{m}}{\sigma^{2}}}.
\]
The upper bound follows from the a.s. inequality
\[
\frac{\partial}{\partial s}{v}_t(s,\lambda, K^{(0)})\le (\gamma^2+\kappa) {v}_t(s,\lambda,K^{(0)})^2 e^{-K^{(0)}_s}.
\]
\end{rem}

It is important to note that when the branching mechanism is of the form $\psi(u)=-\alpha u+c_\beta u^{\beta+1}$ for $\beta\in (0,1]$, one can deduce directly from (\ref{bplaplace2}) and (\ref{bplaplace3}) by taking $\lambda$ and $t$ to $\infty$, and (\ref{BerYor1}) that
\[
\lim_{t\to\infty}{Z_t}=0,\quad \textrm{a.s.,}\qquad \textrm{for }\quad \mathbf{m}\le 0. 
\]
Similarly,  using (\ref{BerYor2}), in the case when $\mathbf{m}>0$ we have
\begin{align*}
\Proconx{z}{\lim_{t\to\infty}Z_t=0}{K}&=\Expo{-z\left(\beta c_\beta\int_0^\infty e^{-\beta(K_u+\alpha u)} \ud u\right)^{-1/\beta}},\qquad \textrm{a.s.,}
\end{align*}
and in particular
\[
\mathbb{P}\Big(W=0\Big)=\mathbb{P}_z\Big(\lim_{t\to\infty}Z_t=0\Big)=\mathbb{E}_z\left[\Expo{-z\left(\beta c_\beta\int_0^\infty e^{-\beta(K_u+\alpha u)} \ud u\right)^{-1/\beta}}\right].
\]
The latter probability can be computed explicitly using (\ref{Dufresne}).

 We finish this section with a remark on the Neveu case. If we take limits as $\lambda\uparrow \infty$, in (\ref{bplaplacekneveu}),  we  obtain that the Neveu CBBRE survives conditionally on the environment, in other words
\begin{align*}
\Proconx{z}{Z_t>0}{K}&=1,
\end{align*}
for all $t\in (0,\infty)$ and $z\in(0,\infty)$. Moreover since the process has c\`adl\`ag paths, we deduce the Neveu CBBRE survives a.s., i.e.
\[
\p_z\Big(Z_t>0, \textrm{ for all } t\ge 0\Big)=1, \qquad \mbox{ for all } z>0.
\]

On the one hand, using integration by parts we obtain
$$\int_0^t e^{-s} K_s\ud s= \sigma \int_0^t e^{-s}\ud B^{(e)}_s-(1-e^{-t})\frac{\sigma^2}{2}-e^{-t}K_t.$$
Since 
$$\left\langle \sigma\int_0^{\cdot} e^{-s}\ud B^{(e)}_s\right\rangle_t=\frac{\sigma^2}{2}(1-e^{-2t})<\infty \qquad\mbox{ and } \qquad \Exp{\left\langle \sigma\int_0^{\cdot} e^{-s}\ud B^{(e)}_s\right\rangle_{\infty}}=\frac{\sigma^2}{2},$$
we have that
$$\underset{t\rightarrow \infty}{\lim}\int_0^t e^{-s} K_s\ud s=\sigma\int_0^{\infty} e^{-s}\ud B^{(e)}_s-\frac{\sigma^2}{2},$$
exists and its law is  Gaussian  with mean $-\frac{\sigma^2}{2}$ and variance $\frac{\sigma^2}{2}$. Hence, if we take  limits as $t\uparrow \infty$ in (\ref{bplaplacekneveu}), we observe 
\[
\e_z\Big[\exp\Big\{-\lambda  \lim_{t\to\infty} Z_t e^{-K_t}\Big\}\Big | K\Big]=\Expo{-z\Expo{\int_0^\infty e^{-s}K_s \ud s}}, \qquad \mbox{ for all } z>0.
\]
Since the right-hand side of the above identity does not depend on $\lambda$, this implies that 
\[
\p_z\Big(\lim_{t\to\infty} Z_t e^{-K_t}=0\Big | K\Big)=\Expo{-z\Expo{\int_0^\infty e^{-s}K_s \ud s}}, \qquad \mbox{ for all } z>0,
\]
and taking expectations in the above identity, we deduce
\[
\p_z\Big(\lim_{t\to\infty} Z_t e^{-K_t}=0\Big)=\e\left[\Expo{-z\Expo{\int_0^\infty e^{-s}K_s \ud s}}\right], \qquad \mbox{ for all } z>0.
\]
In conclusion, the Neveu process is conservative and survives a.s., but  the extinction probability is given by the Laplace transform of a log-normal distribution.

Finally, if we multiply (\ref{bplaplacekneveu}) by $e^{K_t}$, differentiate with respect to  $\lambda $ and then we take expectations on both sides,  we deduce
\[
\begin{split}
\e_{z}&\Big[Z_t \exp\Big\{-\lambda Z_t e^{-K_t}\Big\}\Big]\\
&\hspace{2cm}=ze^{-t}\lambda^{e^{-t}-1}\e\left[\Expo{K_t+\int_0^t e^{-s}K_s \ud s-z\lambda^{e^{-t}}\Expo{\int_0^t e^{-s}K_s \ud s}}\right].
\end{split}
\]
Taking limits as $\lambda$ goes to $0$, it is clear
$$\Expx{z}{Z_t}=\infty, \qquad \textrm{for }\quad t>0,\ z>0.$$

\section{Stable case}

The stable case is  perhaps one of the most interesting examples  of CB-processes. One of the advantages of this class of  CB-processes is that we can perform explicit computations of many functionals, see for instance \cite{Berfont, kypa, Lambert}, and that they appear in many other areas of probability such as coalescent theory, fragmentation theory,  L\'evy trees and self-similar Markov process to name but a few.  As we will see  below, we can also perform a lot of explicit computations when the stable CB-process is affected by a Brownian random environment.

 In the sequel, we shall assume that the branching mechanism satisfies $\psi(u)=-\alpha u +c_\beta u^{\beta+1},$ for  $u \geq 0,$ and $\beta\in(-1,0)\cup(0,1).$ Recall that $\alpha\in \mathbb{R}$, and 
\[
\left\{ \begin{array}{ll}
 c_\beta<0 &\textrm{ if $\beta \in(-1,0),$}\\
c_\beta>0 & \textrm{ if $\beta \in (0,1).$}
 \end{array} \right .
\]
B\"oinghoff and Hutzenthaler  \cite{CbMh} studied the particular case when  $\beta=1$, also known as the Feller diffusion case. The authors in \cite{CbMh} gave a precise asymptotic behaviour for the survival probability and also studied the so-called $Q$-process. In this section, we prove similar results for the case when $\beta\in(0,1)$ and we obtain new results on the asymptotic behaviour of non-explosion for the case when $\beta\in(-1, 0)$. 

Recall from identity (\ref{bplaplace3}) that the stable CBBRE $Z=(Z_t, t\ge 0)$, satisfies
\begin{align*}
\e_z\Big[\exp\Big\{-\lambda Z_t e^{-(K_t+\alpha t)}\Big\}\Big| K\Big]=\Expo{-z\left(\lambda^{-\beta}+\beta c_\beta\int_0^t e^{-\beta(K_u+\alpha u)} \ud u\right)^{-1/\beta}}.
\end{align*}
If we take limits as $\lambda$ goes to $\infty$, in the above identity  we  obtain for all $z, t> 0$,
\begin{align}\label{eqabs}
\Proconx{z}{Z_t>0}{K}&=1-\Expo{-z\left(\beta c_\beta\int_0^t e^{-\beta(K_u+\alpha u)} \ud u\right)^{-1/\beta}}\mathbf{1}_{\{\beta>0\}},\qquad \textrm{a.s.,}
\end{align}
where the same holds true for the Feller case by taking $\beta=1$ and $c_\beta=\gamma^2$. 

On the other hand, if we take limits as $\lambda$ goes to $0$, we deduce that $t,z>0$
\begin{align*}
\mathbb{P}_z\Big( Z_t < \infty\Big|K\Big)=\Expo{-z\left(\beta c_\beta\int_0^t e^{-\beta(K_u+\alpha u)} \ud u\right)^{-1/\beta}}\mathbf{1}_{\{\beta<0\}}+\mathbf{1}_{\{\beta>0\}}\qquad \textrm{a.s.}
\end{align*}
It  is then  clear that if $\beta\in (-1,0)$, then the survival probability equals 1, for all $t\geq 0$. If $\beta\in (0, 1]$ then the process is conservative.

From the above identities, a natural question arises: \textit{can we determine the  asymptotic behaviour, when $t$ goes to $\infty$, of $\p_z(Z_t>0)$ for  $\beta\in (0, 1]$ and  $\p_z(Z_t<\infty)$ for  $\beta\in (-1, 0)$?}  We observe below that the answer of this question depends on a fine study of  the asymptotic behaviour of 
\[
\e\left[\Expo{-z\left(\beta c_\beta\int_0^t e^{-\beta(K_u+\alpha u)} \ud u\right)^{-1/\beta}}\right].
\]
For this purpose, let us recall some interesting facts of the exponential functional of a Brownian motion with drift.
\subsection{Exponential functional of a Brownian motion with drift}
 In what follows, the following functional will be of particular interest. Let $I_{t}^{(\eta)}$ be the exponential functional of a  Brownian motion with drift $\eta\in \R$, in other words
\begin{align*}
I_{t}^{(\eta)}:=\int_{0}^{t} \exp\Big\{2(\eta s + B_{s})\Big\}\ud s, \qquad t\in [0,\infty).
\end{align*}
The law of such random variable have been deeply studied by many authors. Up to our knowledge this is the unique example for which there exist an explicit formula for the joint distribution of $(I_{t}^{(\eta)},B_t+\eta t)$, see for instance  Proposition 2 in Matsumoto and Yor \cite{MaYor}. In particular,  for all $t,\ u \in (0,\infty)$ and $x\in \R$, we have
\begin{align*} 
\Procon{I_{t}^{(\eta)}\in \ud u}{B_t+\eta t=x}=\frac{\sqrt{2\pi t}}{u}\Expo{\frac{x^{2}}{2t}}\Expo{-\frac{1}{2u}\Big(1+e^{2x}\Big)}\theta_{e^{x}/u}(t)\ud u,
\end{align*}
 where
\begin{align*}
\theta_{r}(t)=\frac{r}{\sqrt{2\pi^3 t}}e^{\frac{\pi^2}{2t}}\int_{0}^{\infty}e^{-\frac{y^{2}}{2t}-r \cosh(y)}\sinh (y)\sin\left(\frac{\pi y}{t}\right)\ud y,\qquad r>0.
\end{align*}

The following lemma generalizes Lemma 4.2 in \cite{CbMh} from $p=1$ to $p\geq 0$.
\begin{lemma} \label{bpmomento} Let  $\eta \in \mathbb{R}$ and $p\ge 0$. Then  for every $t >0$, we have
\begin{align*}
 i)\hspace{2cm} \mathbb{E} \left[  \left( I_{t}^{(\eta)} \right)^{-p} \right] & =   e^{(2p^2-2p\eta) t}   \mathbb{E} \left[  \left( I_{t}^{(-(\eta-2p))}\right)^{-p} \right],\\
 ii)\hspace{1.6cm} \mathbb{E} \left[  \left( I_{t}^{(\eta)}\right)^{-2p} \right] & \leq  e^{(2p^2-2p\eta) t} \mathbb{E} \left[  \left( I_{t/2}^{(-(\eta-2p))}\right)^{-p} \right]\mathbb{E} \left[  \left( I_{t/2}^{((\eta-2p))}\right)^{-p} \right].
\end{align*}
\end{lemma}
\begin{proof} Using the time reversal property for Brownian motion, we observe that the process $(\eta t+B_t-\eta(t-s )-B_{t-s}, 0\le s\le t )$  has the same law as $(\eta s+B_s, 0\le s\le t)$. Then, we deduce  that 
\[
\int_0^t e^{2(\eta s+B_s)}{\rm d s} \quad\textrm{has  the same law as } \quad e^{2(\eta t+B_t)}\int_0^t e^{-2(\eta s+B_s)}{\rm d s}.
\]
We now introduce the exponential change of measure known as the Esscher transform or Girsanov's formula
\[
\frac{{\rm d} \mathbb{P}^{(\lambda)}}{{\rm d}\mathbb{P}}\bigg|_{\mathcal{F}_t}= e^{\lambda B_t -\frac{\lambda^2}{2}t}, \qquad \textrm{for }\quad \lambda \in \mathbb{R},
\]
where $(\mathcal{F}_t)_{t\ge 0}$ is the natural filtration generated by the Brownian motion $B$ which is naturally completed. Observe that under $\mathbb{P}^{(\lambda)}$, the process $B$ is a Brownian motion with drift $\lambda$. Hence, taking $\lambda=-2p$, we deduce
\[
\begin{split}
\mathbb{E}\left[\left(\int_0^t e^{2(\eta s+B_s)}{\rm d s}\right)^{-p}\right]&=\mathbb{E}\left[e^{-2p(\eta t+B_t)}\left(\int_0^t e^{-2(\eta s+B_s)}{\rm d s}\right)^{-p}\right]\\
&=e^{-2p\eta t} e^{2p^2 t}\mathbb{E}^{(-2p)}\left[\left(\int_0^t e^{-2(\eta s+B_s)}{\rm d s}\right)^{-p}\right]\\
&=e^{-2p\eta t} e^{2p^2 t}\mathbb{E}\left[\left(\int_0^t e^{-2((\eta-2p) s+B_s)}{\rm d s}\right)^{-p}\right],
\end{split}
\]
which implies the first identity, thanks to the symmetry property of Brownian motion.

In order to get the second identity, we observe 
\[
\int_0^t e^{2(\eta s+B_s)}{\rm d s}=\int_0^{t/2} e^{2(\eta s+B_s)}{\rm d s} +e^{\eta t +2B_{t/2}}\int_0^{t/2} e^{2(\eta s+\tilde{B}_s)}{\rm d s},
\]
where $\tilde{B}_{s}={B}_{s+{t/2}}-B_{t/2}$,  $s\ge 0$, is a Brownian motion which is independent of  $(B_u, 0\le u\le t/2)$. Therefore, using part (i), we deduce
\[
\begin{split}
\mathbb{E}\left[\left(\int_0^t e^{2(\eta s+B_s)}{\rm d s}\right)^{-2p}\right]&\\
&\hspace{-2cm}\le \mathbb{E}\left[\left(e^{\eta t +2B_{t/2}}\int_0^{t/2} e^{2(\eta s+B_s)}{\rm d s}\right)^{-p}\right] \mathbb{E}\left[\left(\int_0^{t/2} e^{2(\eta s+B_s)}{\rm d s}\right)^{-p}\right]\\
&\hspace{-2cm}\le e^{(p^2-\eta p)t} \mathbb{E}\left[\left(e^{\eta t +2B_{t/2}}\int_0^{t/2} e^{2(\eta s+B_s)}{\rm d s}\right)^{-p}\right] \mathbb{E}\left[\left( I_{t/2}^{(-(\eta-2p))}\right)^{-p} \right].\\
\end{split}
\]
On the other hand from the Esscher transform with $\lambda=-2p$, we get
\[
\begin{split}
\mathbb{E}\left[\left(e^{\eta t +2 B_{t/2}}\int_0^{t/2} e^{2(\eta s+B_s)}{\rm d s}\right)^{-p}\right]&=e^{-p\eta t} e^{p^2 t}\mathbb{E}^{(-2p)}\left[\left(\int_0^{t/2} e^{2(\eta s+B_s)}{\rm d s}\right)^{-p}\right]\\
&=e^{-p\eta t} e^{p^2 t}\mathbb{E}\left[\left(\int_0^{t/2} e^{2((\eta-2p) s+B_s)}{\rm d s}\right)^{-p}\right].
\end{split}
\]
Putting all the pieces together, we deduce
\[
\begin{split}
\mathbb{E}\left[\left(\int_0^t e^{2(\eta s+B_s)}{\rm d s}\right)^{-2p}\right]&\le e^{(2p^2-\eta 2p)t} \mathbb{E}\left[\left( I_{t/2}^{((\eta-2p))}\right)^{-p} \right]\mathbb{E}\left[\left( I_{t/2}^{(-(\eta-2p))}\right)^{-p} \right].\\
\end{split}
\]
The proof of the lemma is now complete.
\end{proof}

We are also interested in 
\[
I_{\infty}^{(\eta)}:=\int_{0}^{\infty} \exp\Big\{2(\eta s + B_{s})\Big\}\ud s, 
\]
which is  finite a.s. whenever $\eta<0$. We recall that according to Dufresne \cite{Du}, 
\begin{equation}\label{dufresne}
I^{(\eta)}_\infty\quad \textrm{ has the same law as }\quad \Big(2\Gamma_{-\eta}\Big)^{-1},
\end{equation}
where $\Gamma_v$ is Gamma r.v. with parameter $v$.
\subsection{Explosion probability}
Throughout this section, we assume that  $\beta\in(-1, 0)$. As we will see in the main result of this section, the asymptotic behaviour of the probability of explosion depends on the value of 
\[
\mathbf{m}=\alpha-\frac{\sigma^2}{2}.
\]
We recall that when $\sigma =0$,  i.e there is no random environment, the stable CB-processes  explodes with positive probability. Moreover, when $\alpha=0$, we can compute explicitly the asymptotic behaviour of the probability of explosion. 

When a Brownian random environment affects the stable CB-process, the process behaves completely different. In fact, it also  explodes with positive probability but  we have three different regimes of the asymptotic behaviour of the non-explosion probability that depends on the parameters of  the random environment. Up to our knowledge, this behaviour was never observed  or studied before. We call these regimes {\it subcritical-explosion}, {\it critical-explosion} or {\it supercritical-explosion} depending on whether this probability stays positive, converge to zero polynomially fast or converges to zero exponentially fast.

 Let
\begin{align*}
\eta:=-\frac{2}{\beta\sigma^2}\mathbf{m} \qquad\textrm{ and }\qquad \mathbf{k}=\left(\frac{\beta \sigma^2}{2c_\beta}\right)^{1/\beta},
\end{align*}
and define
$$ g(x):=\Expo{- \mathbf{k} x^{1/\beta}}, \qquad \textrm{for}\qquad x\ge 0.$$
From identity (\ref{eqexplo}) and the scaling property, we deduce  for $\beta\in (-1,0)$ and $\eta>-1$, 
\begin{align}\label{bpZyAex}
\Prox{z}{Z_{t}<\infty}=\Exp{g\left(\frac{z^{\beta}}{2I_{\beta^2\sigma^{2}t/4}^{(\eta)}}\right)}=\int_{0}^{\infty}g(z^{\beta}v)p_{t\sigma_{e}^{2}/4,\eta}(v)\ud v,
\end{align}
where $p_{\nu,\eta}$ denotes the density function of $1/2I_{\nu}^{(\eta)}$ which according to Matsumoto and Yor \cite{MaYor}, satisfies
\begin{align}\label{idMY}
p_{\nu,\eta}(x)=&\frac{e^{-\eta^{2}\nu/2}e^{\pi^{2}/2\nu}}{\sqrt{2}\pi^{2}\sqrt{\nu}}\Gamma\left(\frac{\eta+2}{2}\right)e^{-x}x^{-(\eta+1)/2}\int_{0}^{\infty}\int_{0}^{\infty}e^{\xi^{2}/2\nu}s^{(\eta-1)/2}e^{-xs}\\
&\hspace{7.5cm}\times \frac{\sinh(\xi)\cosh(\xi)\sin(\pi\xi/\nu)}{(s+\cosh(\xi)^{2})^{\frac{\eta+2}{2}}}\ud \xi \ud s. \nonumber
\end{align}
We also denote 
\[
\mathcal{L}_{\eta, \beta}(\theta)=\e\Big[ e^{-\theta \Gamma_{-\eta}^{1/\beta}}\Big], \qquad \textrm{for}\quad \theta\ge 0.
\]

\begin{theorem}\label{bplimitesnegativo}
Let $(Z_{t}, t\geq 0)$ be the stable CBBRE with index $\beta\in(-1,0)$ defined by the SDE (\ref{csbpbrest}) with $Z_0=z>0$. 
\begin{itemize}
\item[i)] Subcritical-explosion. If $\mathbf{m}<0$, then
\begin{equation*}
\underset{t\rightarrow\infty}{\lim}\Prox{z}{Z_t<\infty}=\mathcal{L}_{\eta, \beta}\left(z\mathbf{k}\right). \label{bpsubex}
\end{equation*}
\item[ii)]  Critical-explosion. If $\mathbf{m}=0$, then
\begin{equation*}
\underset{t\rightarrow\infty}{\lim}\sqrt{t}\ \Prox{z}{Z_t<\infty}=- \frac{\sqrt{2}}{\sqrt{\pi}\beta\sigma}\int_0^{\infty}e^{-z\mathbf{k}x^{1/\beta}-x}\frac{\ud x}{x}.  \label{bpcriticoex}
\end{equation*}
\item[iii)] Supercritical-explosion. If $\mathbf{m}>0$, then
\begin{equation*}
\underset{t\rightarrow\infty}{\lim}t^{\frac{3}{2}} e^{\frac{\mathbf{m}^2 t}{2\sigma^2}}\Prox{z}{Z_t<\infty}=- \frac{8}{\beta^3\sigma^3}\int_{0}^{\infty}g(z^{\beta}v)\phi_{\eta}(v)\ud v, \label{bpsuperex}
\end{equation*}
where
 $$\phi_{\eta}(v)=\int_{0}^{\infty}\int_{0}^{\infty}\frac{1}{\sqrt{2}\pi}\Gamma\left(\frac{\eta+2}{2}\right)e^{-v}v^{-\eta/2}u^{(\eta-1)/2}e^{-u}\frac{\sinh(\xi)\cosh(\xi)\xi}{(u+v\cosh(\xi)^{2})^{\frac{\eta+2}{2}}}\ud \xi \ud u. $$
\end{itemize}
\end{theorem}

\begin{proof} Our arguments follows from similar reasoning as in the proof of Theorem 1.1 in  B\"oinghoff and Hutzenthaler \cite{CbMh}. For this reason, following the same notation as in \cite{CbMh}, we just provide the fundamental ideas of the proof. 

The subcritical-explosion case (i) follows from the identity in law by Dufresne (\ref{dufresne}). More precisely, 
 from  (\ref{dufresne}), (\ref{bpZyAex}) and the Dominated Convergence Theorem, we deduce 
\begin{align*}
\underset{t\rightarrow\infty}{\lim}\Prox{z}{Z_t<\infty}=\Exp{\Expo{-z\mathbf{k}\Gamma_{-\eta}^{1/\beta}}}=\mathcal{L}_{\eta, \beta}\left(z\mathbf{k}\right).
\end{align*}
In order to prove the critical-explosion case (ii), we use Lemma 4.4 in  \cite{CbMh}. From identity (\ref{bpZyAex}) and   applying Lemma 4.4 in \cite{CbMh} to
\begin{equation}\label{boingex}
g(z^{\beta}x)=\Expo{-z\mathbf{k}x^{1/\beta}}\leq \frac{x^{-1/\beta}}{z\mathbf{k}}, \qquad\textrm{for} \quad x\geq 0,
\end{equation}
we get 
\begin{align*}
\underset{t\rightarrow\infty}{\lim}\sqrt{t}\ \Prox{z}{Z_t<\infty}&=-\frac{2}{\beta\sigma}\ \underset{t\rightarrow\infty}{\lim}\sqrt{\frac{t\sigma^2\beta^2}{4}}\ \Exp{g\left(\frac{z^{\beta}}{2I_{\beta^2\sigma^{2}t/4}^{(\eta)}}\right)}\\
&=-\frac{\sqrt{2}}{\sqrt{\pi}\beta\sigma}\int_0^{\infty}e^{-z\mathbf{k}x^{1/\beta}-x}\frac{\ud x}{x},
\end{align*}
which is finite since the inequality (\ref{boingex}) holds.

We now consider the supercritical-explosion case (iii).  Observe that for all $n\ge 0$,
\begin{equation*}
g(z^{\beta}x)=\Expo{-z\mathbf{k}x^{1/\beta}}\leq \frac{x^{-n/\beta}}{n! (z\mathbf{k})^n }, \qquad\textrm{for}\quad  x\geq 0.
\end{equation*}
Therefore
 using the above inequality for a fixed $n$, Lemma 4.5 in  \cite{CbMh} and identity  (\ref{bpZyAex}),  we obtain that for each  $0<\mathbf{m}$ that satisfies $\mathbf{m}<n\sigma^2/2$, the following limit holds 
\begin{align*}
\underset{t\rightarrow\infty}{\lim}t^{3/2}\ e^{\mathbf{m}^2 t/2\sigma^2}\Prox{z}{Z_{t}>0}&=\underset{t\rightarrow\infty}{\lim}t^{3/2}\ e^{\eta^2 \beta^2\sigma^{2}t/8}\Exp{g\left(\frac{z^{\beta}}{2I_{\beta^2\sigma^{2}t/4}^{(\eta)}}\right)}\\
&=-\frac{8}{\beta^3\sigma^3}\int_{0}^{\infty}e^{-z\mathbf{k}v^{1/\beta}}\phi_{\eta}(v)\ud v,
\end{align*}
where $\phi_\eta$ is defined as in the statement of the Theorem. Since this limit holds for any $n\ge 1$, we deduce that it must hold for $\mathbf{m}>0$. This completes the proof.

\end{proof} 

\subsection{Survival probability}
Throughout this section, we assume that  $\beta\in(0, 1)$. One of the aims of this section is to compute the asymptotic behaviour of the survival probability  and we will see that it depends on the value of $\mathbf{m}$. We find five different regimes as in the Feller case (see for instance Theorem 1.1 in \cite{CbMh}) and CB-processes with catastrophes (see for instance Proposition 5 in \cite{Bapa}). Recall that in the classical theory of branching processes, the survival probability stays positive, converges to zero polynomially fast or converges to zero exponentially fast, depending of whether  the process is \textit{supercritical} ($\mathbf{m}>0$), \textit{critical} ($\mathbf{m}=0$) or \textit{subcritical} ($\mathbf{m}<0$), respectively. In the stable CBBRE there is another phase transition in the subcritical regime. This phase transition occurs when $\mathbf{m}=-\sigma^2$. We say that the stable CBBRE is \textit{weakly subcritical} if $-\sigma^2<\mathbf{m}<0$, \textit{intermediately subcritical} if $\mathbf{m}=-\sigma^2$ and \textit{strongly subcritical} if $\mathbf{m}<-\sigma^2$.

 Recall that  
\begin{align*}
\eta=-\frac{2}{\beta\sigma^2}\mathbf{m}\qquad\textrm{and}\qquad \mathbf{k}=\left(\frac{\beta\sigma^2}{2c_\beta}\right)^{1/\beta},
\end{align*}
and define
$$ f(x):=1-\Expo{-\mathbf{k} x^{1/\beta}}, \qquad \textrm{for}\qquad x\ge 0.$$
From identity (\ref{eqabs}) and the scaling property, we deduce  for $\beta>0$ and $\eta>-1$, 
\begin{align}\label{bpZyA}
\Prox{z}{Z_{t}>0}=\Exp{f\left(\frac{z^{\beta}}{2I_{\beta^2\sigma^{2}t/4}^{(\eta)}}\right)}=\int_{0}^{\infty}f(z^{\beta}v)p_{\beta^2\sigma^{2}t/4,\eta}(v)\ud v,
\end{align}
where $p_{\nu,\eta}$ denotes the density function of $1/2I_{\nu}^{(\eta)}$ and is given in (\ref{idMY}).

\begin{theorem}\label{bplimites}
Let $(Z_{t}, t\geq 0)$ be the stable CBBRE with index $\beta\in(0,1)$ defined by the SDE (\ref{csbpbrest}) with $Z_0=z>0$. \begin{itemize}
\item[i)] Supercritical. If $\mathbf{m}>0$, then
\begin{equation}
\underset{t\rightarrow\infty}{\lim}\Prox{z}{Z_t>0}=1-\underset{n=0}{\overset{\infty}{\sum}}\frac{(-z\mathbf{k})^{n}}{n!} \frac{\Gamma(\frac{n}{\beta}-\eta)}{\Gamma(-\eta)}. \label{bpsuper}
\end{equation}
\item[ii)]  Critical. If $\mathbf{m}=0$, then
\begin{equation}
\underset{t\rightarrow\infty}{\lim}\sqrt{t}\ \Prox{z}{Z_t>0}=- \frac{\sqrt{2}}{\sqrt{\pi}\beta\sigma}\underset{n=1}{\overset{\infty}{\sum}}\frac{(-z\mathbf{k})^{n}}{n!} \Gamma\left(\frac{n}{\beta}\right).  \label{bpcritico}
\end{equation}
\item[iii)] Weakly subcritical. If $\mathbf{m}\in (-\sigma^2, 0)$, then
\begin{equation}
\underset{t\rightarrow\infty}{\lim}t^{\frac{3}{2}} e^{\frac{\mathbf{m}^2 t}{2\sigma^2}}\Prox{z}{Z_t>0}= \frac{8}{\beta^3\sigma^3}\int_{0}^{\infty}f(z^\beta v)\phi_{\eta}(v)\ud v, \label{bpsubdebil}
\end{equation}
where
 $$\phi_{\eta}(v)=\int_{0}^{\infty}\int_{0}^{\infty}\frac{1}{\sqrt{2}\pi}\Gamma\left(\frac{\eta+2}{2}\right)e^{-v}v^{-\eta/2}u^{(\eta-1)/2}e^{-u}\frac{\sinh(\xi)\cosh(\xi)\xi}{(u+v\cosh(\xi)^{2})^{\frac{\eta+2}{2}}}\ud \xi \ud u. $$
\item[iv)] Intermediately subcritical. If $\mathbf{m}=-\sigma^2$, then
\begin{equation}
\underset{t\rightarrow\infty}{\lim}\sqrt{t} e^{\sigma^2 t/2}\Prox{z}{Z_t>0}=z \frac{\sqrt{2}}{\sqrt{\pi}\beta\sigma}\mathbf{k}\Gamma\left(\frac{1}{\beta}\right). \label{bpsubinter}
\end{equation}
\item[v)] Strongly subcritical. If $\mathbf{m}<-\sigma^2$, then
\begin{equation}
\underset{t\rightarrow\infty}{\lim} e^{-\frac{1}{2} (2\mathbf{m}+\sigma^2)t}\Prox{z}{Z_t>0}=z \mathbf{k}\frac{\Gamma(\eta-1/\beta)}{\Gamma(\eta-2/\beta)}. \label{bpsubfuerte}
\end{equation}
\end{itemize}
\end{theorem}

\begin{proof} As in the proof of Theorem \ref{bplimitesnegativo}, and following the same notation as in \cite{CbMh}, we just provide the fundamental ideas of the proof. 

The supercritical case (i) follows from the identity in law by Dufresne (\ref{dufresne}). More precisely, 
 from  (\ref{dufresne}), (\ref{bpZyA}) and the Dominated Convergence Theorem, we deduce 
\begin{align*}
\underset{t\rightarrow\infty}{\lim}\Prox{z}{Z_t>0}=\Exp{1-\Expo{-z\mathbf{k}\Gamma_{-\eta}^{1/\beta}}}=1-\underset{n=0}{\overset{\infty}{\sum}}(-z\mathbf{k})^{n} \frac{\Gamma(\frac{n}{\beta}-\eta)}{n!\Gamma(-\eta)}.
\end{align*}
In order to prove the critical case (ii), we use Lemma 4.4 in  \cite{CbMh}. From identity (\ref{bpZyA}) and   applying Lemma 4.4 in \cite{CbMh} to
\begin{equation}\label{boing}
g(x):=1-\Expo{-z\mathbf{k}x^{1/\beta}}\leq z\mathbf{k}x^{1/\beta}, \quad x\geq 0,
\end{equation}
we get 
\begin{align*}
\underset{t\rightarrow\infty}{\lim}\sqrt{t}\ \Prox{z}{Z_t>0}&=\frac{2}{\beta\sigma}\ \underset{t\rightarrow\infty}{\lim}\sqrt{\frac{t\sigma^2\beta^2}{4}}\ \Exp{1-\Expo{-z\mathbf{k}\left(2 I^{(\eta)}_{\beta^2\sigma^{2}t/4}\right)^{-1/\beta}}}\\
&=\frac{\sqrt{2}}{\sqrt{\pi}\beta\sigma}\int_0^{\infty}\Big(1-\Expo{-z\mathbf{k}x^{1/\beta}}\Big)\frac{e^{-x}}{x}\ud x.
\end{align*}
By Fubini's theorem, it is easy to show that, for all $q\geq 0$
\[
\int_0^{\infty}\left(1-e^{-qx^{1/\beta}}\right)\frac{e^{-x}}{x}\ud x
=-\underset{n=1}{\overset{\infty}{\sum}}\frac{(-1)^n}{n!}\Gamma\left(\frac{n}{\beta}\right)q^{n},
\]
which implies (\ref{bpcritico}).

We now consider the weakly subcritical case (iii).  Recall that  inequality (\ref{boing}) still holds, then
 using Lemma 4.5 in  \cite{CbMh} and identity  (\ref{bpZyA}),  we obtain
\begin{align*}
\underset{t\rightarrow\infty}{\lim}t^{3/2}\ e^{\mathbf{m}^2 t/2\sigma^2}\Prox{z}{Z_{t}>0}&=\underset{t\rightarrow\infty}{\lim}t^{3/2}\ e^{\eta^2 \beta^2\sigma^{2}t/8}\Exp{g\left(\frac{z^{\beta}}{2I_{\beta^2\sigma^{2}t/4}^{(\eta)}}\right)}\\
&=\frac{8}{\beta^3\sigma^3}\int_{0}^{\infty}\left(1-\Expo{-z\mathbf{k}v^{1/\beta}}\right)\phi_{\eta}(v)\ud v,
\end{align*}
where $\phi_\eta$ is defined as in the statement of the Theorem.

In the remaining two cases we will use Lemma 4.1 in \cite{CbMh} and Lemma \ref{bpmomento}. For the intermediately subcritical case (iv), we observe that $\eta=2/\beta$. Hence, applying Lemma \ref{bpmomento}   with $p=1/\beta$, we get
$$\mathbb{E} \left[  \left( I_{t}^{(\eta)} \right)^{-1/\beta} \right] =   e^{-\frac{2}{\beta^{2}}t}  \mathbb{E} \left[  \left( I_{t}^{(0)} \right)^{-1/\beta} \right] \quad\mbox{and} \quad
\mathbb{E} \left[  \left( I_{t}^{(\eta)} \right)^{-2/\beta} \right] \leq   e^{-\frac{2}{\beta^{2}}t} \mathbb{E}  \left[ \left( I_{t/2}^{(0)} \right)^{-1/\beta} \right]^{2}.$$
Now, applying  Lemma 4.4 in \cite{CbMh}, we deduce
$$\underset{t\rightarrow\infty}{\lim} \sqrt{t}\ \Exp{\left( 2 I_{t}^{(\eta)} \right)^{-1/\beta}}=\int_{0}^{\infty}\frac{1}{\sqrt{2\pi}}\frac{e^{-a}}{a}a^{1/\beta}\ud a=\frac{1}{\sqrt{2\pi}}\Gamma\left(\frac{1}{\beta}\right),$$
and 
\[
\underset{t\rightarrow\infty}{\lim} \sqrt{t}\ \Exp{\left(  2I_{t}^{(\eta)} \right)^{-2/\beta}}=0.
\] 
Therefore, we can apply  Lemma 4.1 in \cite{CbMh}  with $c_t=\sqrt{t}\ e^{2t/\beta^2}$ and $Y_t=\left(2 I_t^{(\eta)}\right)^{-1/\beta}$, and obtain
\begin{align*}
\underset{t\rightarrow\infty}{\lim}\sqrt{t} e^{\sigma^2 t/2}\Prox{z}{Z_t>0}&=\underset{t\rightarrow\infty}{\lim}\sqrt{t} e^{\sigma^2  t/2}\Exp{1-\Expo{-z\mathbf{k}\left( 2I_{\beta^2\sigma^{2}t/4}^{(\eta)} \right)^{-1/\beta}}}\\
&=\underset{t\rightarrow\infty}{\lim}\sqrt{t} e^{\sigma^2 t/2}z\mathbf{k}
\Exp{\left( 2I_{\beta^2\sigma^{2}t/4}^{(\eta)} \right)^{-1/\beta}}\\
&=z \frac{\sqrt{2}}{\sqrt{\pi}\beta\sigma}\mathbf{k}\Gamma\left(\frac{1}{\beta}\right).
\end{align*}
Finally for the strongly subcritical case, we use again Lemma \ref{bpmomento} with $p=1/\beta$. First observe that $\eta-2/\beta>0$. Thus,  the Monotone Convergence Theorem and the identity of Dufresne (\ref{dufresne}) yield 
\begin{align*}
\underset{t\rightarrow\infty}{\lim} \Exp{\left( 2I_{t/2}^{(-(\eta-2/\beta))}\right)^{-1/\beta}}&=\Exp{\left(2I_{\infty}^{(-(\eta-2/\beta))}\right)^{-1/\beta}}=\Exp{(\Gamma_{\eta-2/\beta})^{1/\beta}}=\frac{\Gamma(\eta-1/\beta)}{\Gamma(\eta-2/\beta)}.
\end{align*}
Since $I_t^{(\eta-2/\beta)}$ goes to $ \infty$ as $t$ increases, from the Monotone Convergence Theorem, we get 
$$\underset{t\rightarrow\infty}{\lim} \Exp{\left( 2I_{t/2}^{(\eta-2/\beta)}\right)^{-1/\beta}}=0.$$
Hence by applying  Lemma 4.1 in \cite{CbMh} with $c_t=e^{-(2/\beta^2-2\eta/\beta) t}$ and $Y_t=\left(2 I_t^{(\eta)}\right)^{-1/\beta}$ we obtain that
\[
\begin{array}{l}
\displaystyle\underset{t\rightarrow\infty}{\lim}e^{-\frac{1}{2} (2\mathbf{m}+\sigma^2)t} \Prox{z}{Z_t>0}\displaystyle=\underset{t\rightarrow\infty}{\lim}c_{\beta^2\sigma^{2}t/4}\Exp{1-\Expo{-\mathbf{k}\left( 2I_{\beta^2\sigma^{2}t/4}^{(\eta)} \right)^{-1/\beta}}}\\
\hspace{4.9cm}\displaystyle=z \mathbf{k}\frac{\Gamma(\eta-1/\beta)}{\Gamma(\eta-2/\beta)}.
\end{array}
\]
This completes the proof.
\end{proof} 

\subsection{Conditioned stable CBBRE}
Here, we are interested in studying two conditioned versions of the stable CBBRE: the process conditioned to be never extinct (or  $Q$-process) and  the process conditioned on eventual extinction.
Our methodology follows similar arguments as those used in Lambert \cite{Lambert} and extend the results obtained by B\"oinghoff and Hutzenthaler \cite{CbMh} and Hutzenthaler \cite{Hu} in the stable case
with $\beta\in(0,1)$.  In particular, we obtain that the supercritical stable CBBRE conditioned on eventual extinction possesses a similar phase transition as the subcritical stable CBBRE. It is important to note that such a phase transition has not  been reported in the discrete case. For the continuous case, it was only observed in \cite{Hu} for $\beta=1$. In contrast with the subcritical regime,  the phase transition is given at $\beta\sigma^2$. We say that the supercritical stable CBBRE  conditioned on eventual extinction is  \textit{weakly supercritical} if $0<\mathbf{m}<\beta\sigma^2$, \textit{intermediately supercritical} if $\mathbf{m}=\beta\sigma^2$, and \textit{strongly supercritical} if $\mathbf{m}>\beta\sigma^2$. 
\subsubsection{The process conditioned to be never extinct}
In order to study the stable CBBRE conditioned to be never extinct, we need the following Lemma.

\begin{lemma}\label{bpintegraldeZ}
For every $t \geq 0$, $Z_t$ is integrable.
\end{lemma}

\begin{proof}
Differentiating  the Laplace transform (\ref{bplaplace3}) of $Z_t$  in $\lambda$  and taking limits   as $\lambda$ goes to $0$, on both sides,  we deduce
$$\Expconx{z}{Z_t }{K} =ze^{\mathbf{m}t+\sigma B_t^{(e)}},$$
which is an integrable random variable.
\end{proof}
 Recall that  
\begin{align*}
\eta=-\frac{2}{\beta\sigma^2}\mathbf{m}\qquad\textrm{and}\qquad \mathbf{k}=\left(\frac{\beta\sigma^2}{2c_\beta}\right)^{1/\beta}.
\end{align*}
We now define the function $U:[0,\infty)\rightarrow (0,\infty)$ as follows
\[
U(z)=
\left\{
\begin{array}{ll}
\vspace{0.2cm}- \frac{\sqrt{2}}{\sqrt{\pi}\beta\sigma}\displaystyle\underset{n=1}{\overset{\infty}{\sum}}\frac{(-z\mathbf{k})^{n}}{n!} \Gamma\left(\frac{n}{\beta}\right)& \mbox{if }  \mathbf{m}=0,\\
\vspace{0.2cm}\frac{8}{\beta^3\sigma^3}\displaystyle\int_{0}^{\infty}\left(1-e^{-z\mathbf{k}v^{1/\beta}}\right)\phi_{\eta}(v)\ud v &  \mbox{if }  \mathbf{m}\in(-\sigma^2,0),\\
\vspace{0.2cm}z \displaystyle\frac{\sqrt{2}}{\sqrt{\pi}\beta\sigma}\mathbf{k}\Gamma\left(\frac{1}{\beta}\right) & \mbox{if } \mathbf{m}=-{\sigma^2},\\
\vspace{0.2cm} z \mathbf{k}\displaystyle\frac{\Gamma(\eta-1/\beta)}{\Gamma(\eta-2/\beta)} & \mbox{if }  \mathbf{m}<-{\sigma^2},
\end{array}
\right.
\]
where the function $\phi_{\eta}$ is given as in Theorem \ref{bplimites}. We also introduce  
\[
\theta:=\theta(\mathbf{m}, \sigma)=\left\{
\begin{array}{ll}
0& \mbox{if } \mathbf{m}= 0,\\
\vspace{0.2cm}\displaystyle\frac{\mathbf{m}^2}{2\sigma^2}&  \mbox{if }  \mathbf{m}\in(-\sigma^2,0),\\
\vspace{0.2cm} -\displaystyle\frac{2\mathbf{m}+\sigma^2}{2}& \mbox{if } \mathbf{m}\le-{\sigma^2}.\\
\end{array}
\right.
\]
Let $(\mathcal{F}_t)_{t\ge 0}$ be the natural filtration generated by  $Z$ and $T_0=\inf\{t\geq 0: Z_t=0\}$ be the extinction time of the process $Z$. 
The next proposition states, in the critical and subcritical cases,  the existence of the $Q$-process.

\begin{proposition}\label{qprocess}Let $(Z_{t}, t\geq 0)$ be the stable CBBRE with index $\beta\in(0,1)$ defined by the SDE (\ref{csbpbrest}) with $Z_0=z>0$. Then for $\mathbf{m}\le 0$:
\begin{enumerate}
	\item[i)]  The conditional laws $\mathbb{P}_z\left(\cdot\mid  T_0>t+s \right)$ converge as $s\rightarrow \infty$ to a limit denoted by $\mathbb{P}_z^\natural,$ in the sense that  for any $t\geq 0$ and $\Lambda\in\mathcal{F}_t$,
	$$\underset{s\rightarrow\infty}{\lim} \mathbb{P}_z\left(\Lambda \mid T_0>t+s\right)=\mathbb{P}_z^\natural\left(\Lambda\right).$$
	\item[ii)] The probability measure $\mathbb{P}^\natural$ can be expressed as an $h$-transform of $\mathbb{P}$ based on the martingale
	$$D_t=e^{\theta t}U(Z_t),$$
	in the sense that
	$$\ud \mathbb{P}_z^\natural\big|_{\mathcal{F}_t}=\frac{D_t}{U(z)}\ud\mathbb{P}_z\big|_{\mathcal{F}_t}.$$
	\end{enumerate}
\end{proposition}

\begin{proof}
We first prove part (i). Let $z, s,t>0$, and $\Lambda\in\F_t$. From   the Markov property, we observe
\begin{align}\label{bpthetalimite}
\pro_z\left(\Lambda\mid T_0>t+s\right)&=\frac{\mathbb{P}_z\Big(\Lambda\cap \{ T_0>t+s\}\Big)}{\Prox{z}{T_0>t+s}}=\Expx{z}{\frac{ \Prox{Z_t}{Z_{s}>0}}{\Prox{z}{Z_{t+s}>0}} \mathbf{1}_{\Lambda} \Ind{ T_0>t}}.
\end{align}
On the other hand, since the mapping $t\mapsto  I_{t}^{(\eta)}$ is increasing and the function $f(x)=1-\Expo{-\mathbf{k} x^{1/\beta}}$ is decreasing, we deduce from (\ref{bpZyA}) and the Markov property that for any $z,y>0$,
\begin{align*}
0\leq \frac{ \Prox{y}{Z_{s}>0}}{\Prox{z}{Z_{t+s}>0}}&= 
\frac{\mathbb{E}\left[1-\Expo{-y\mathbf{k}\left(  2I_{\sigma^{2}\beta^2 s/4}^{(\eta)}\right)^{-1/\beta}}\right]} {\mathbb{E}\left[1-\Expo{-z\mathbf{k}\left( 2I_{\sigma^{2}\beta^2 (t+s)/4}^{(\eta)}\right)^{-1/\beta}}\right]}\\&\leq \frac{y\mathbf{k} \Exp{ \left(  2I_{\sigma^{2}\beta^2 s/4}^{(\eta)}\right)^{-1/\beta}}} {\mathbb{E}\left[1-\Expo{-z\mathbf{k}\left(  2I_{\sigma^{2}\beta^2 s/4}^{(\eta)}\right)^{-1/\beta}}\right]}.
\end{align*}
Moreover, since $ I_{t}^{(\eta)}$ diverge as $t$ goes to $ \infty$, we have 
$$ \frac{z\mathbf{k}\Exp{ \left(  2I_{\sigma^{2}\beta^2 s/4}^{(\eta)}\right)^{-1/\beta}}}{2}\leq\Exp{ 1-\Expo{-z\mathbf{k}\left(  2I_{\sigma^{2}\beta^2 s/4}^{(\eta)}\right)^{-1/\beta}}}\le z\mathbf{k}\Exp{ \left(  2I_{\sigma^{2}\beta^2 s/4}^{(\eta)}\right)^{-1/\beta}},
$$
for $s$ sufficiently large.
Then for any $s$ greater than some bound chosen independently of $Z_t(\omega)$, we necessarily have
$$0\leq \frac{ \Prox{Z_t}{Z_{s}>0}}{\Prox{z}{Z_{t+s}>0}}\leq \frac{2}{z}Z_t.$$
Now, from the asymptotic behaviour (\ref{bpcritico}), (\ref{bpsubdebil}), (\ref{bpsubinter}) and (\ref{bpsubfuerte}), we get
\[
\underset{s\rightarrow\infty}{\lim}\frac{ \Prox{Z_t}{Z_{s}>0}}{\Prox{z}{Z_{t+s}>0}}=\frac{e^{\theta t}U(Z_t)}{U(z)}.
\]
Hence,  Dominated Convergence and identity (\ref{bpthetalimite}) imply
\begin{equation}\label{idq-proc}
\underset{s\rightarrow\infty}{\lim} \mathbb{P}_z\left(\Lambda \mid T_0>t+s\right)=\mathbb{E}_z\left[\frac{e^{\theta t}U(Z_t)}{U(z)}\mathbf{1}_{ \Lambda} \right].
\end{equation}

Next, we prove part (ii). In order to do so, we use (\ref{idq-proc}) with $\Lambda=\Omega$ to deduce
\[
\mathbb{E}_z\left[e^{\theta t}U(Z_t) \right]=U(z).
\]
Therefore, from the Markov property, we obtain 
\[
\mathbb{E}_z\left[e^{\theta (t+s)}U(Z_{t+s})\Big|\mathcal{F}_s \right]=e^{\theta s}\mathbb{E}_{Z_s}\left[e^{\theta t}U(Z_{t}) \right]=e^{\theta s}U(Z_s),
\]
 implying that $D$ is a martingale. 
 \end{proof}
\subsubsection{The process conditioned on eventual extinction}
Here, we assume that $\mathbf{m}>0$ and define for $z>0$.
\[
U_\ast(z):=\displaystyle\underset{n=0}{\overset{\infty}{\sum}}\frac{(-z\mathbf{k})^{n}}{n!}\frac{\Gamma(n/\beta-\eta)}{\Gamma(-\eta)}.
\]
In the supercritical case, we are interested in the process conditioned on eventual extinction.
\begin{proposition}Let $(Z_{t}, t\geq 0)$ be the stable CBBRE with index $\beta\in(0,1)$ defined by the SDE (\ref{csbpbrest}) with $Z_0=z>0$. Then for $\mathbf{m}>0$, the conditional law 
\[
\mathbb{P}_z^\ast(\cdot)=\mathbb{P}_z\left(\cdot\mid  T_0<\infty \right),
\]
satisfies for any $t\geq 0$,
$$\ud \mathbb{P}_z^\ast\big|_{\mathcal{F}_t}=\frac{U_\ast(Z_t)}{U_\ast(z)}\ud\mathbb{P}_z\big|_{\mathcal{F}_t}.$$
Moreover, $(U_\ast(Z_t), t\ge 0)$ is a martingale.
\end{proposition}
\begin{proof}
Let $z,\, t\geq 0$ and $\Lambda\in \mathcal{F}_t$, then 
\begin{align*}
\mathbb{P}_z\left(\Lambda \Big| T_0<\infty\right)= \frac{\mathbb{P}_z\left(\Lambda\cap\{ T_0<\infty\}\right)}{\mathbb{P}_z\left(T_0<\infty\right)}=\underset{s\rightarrow\infty}{\lim}\frac{\mathbb{P}_z\left(\Lambda\cap\{Z_{t+s}=0\}\right)}{U_\ast(z)}.\end{align*}
On the other hand,   the Markov property implies
\[
\mathbb{P}_z\left(\Lambda\cap\{ Z_{t+s}=0\}\right)=\mathbb{E}_z\left[\mathbb{P}_z\Big(Z_{t+s}=0\Big| \mathcal{F}_t\Big)\mathbf{1}_{\Lambda}\right]=\mathbb{E}_z\left[\mathbb{P}_{Z_t}\Big(Z_{s}=0\Big)\mathbf{1}_{\Lambda}\right].
\]
Therefore using the Dominated Convergence Theorem, we deduce
\begin{align*}
\mathbb{P}_z\left(\Lambda \Big| T_0<\infty\right)=\underset{s\rightarrow\infty}{\lim}\frac{\mathbb{E}_z\left[\mathbb{P}_{Z_t}\Big(Z_{s}=0\Big)\mathbf{1}_{\Lambda}\right]}{U_\ast(z)}=\frac{\mathbb{E}_z\Big[U_\ast(Z_t)\mathbf{1}_{\Lambda}\Big]}{U_\ast(z)}.
\end{align*}
The proof that $(U_\ast(Z_t), t\ge0)$ is a martingale follows from the same argument as in the proof of part (ii) of Proposition \ref{qprocess}.
\end{proof}

Observe that $\mathbb{P}_z^\ast(Z_t>0)$ goes to 0 as $t\rightarrow \infty$.  Hence a natural problem to study is the rates of convergence of  the survival probability of the CBBRE conditioned on eventual extinction.  As we mentioned before,  we  obtain a   phase transition which is similar to the subcritical regime. 

It is important to note that the arguments that we will use below also provides the rate of convergence of the inverse of  exponential functionals of a Brownian motion with drift towards its limit,  the  Gamma random variable.
The latter comes from the following observation. Since $U_*(Z_t)$ is a martingale, we deduce
\begin{equation*}
\begin{split}
\mathbb{P}_z^\ast(Z_t>0)=&\mathbb{P}_z\left(Z_t>0\mid  T_0<\infty \right)=\frac{1}{U_{\ast}(z)}\left({U_{\ast}(z)}-{\Prox{z}{Z_t=0}}\right)\\
=&\frac{1}{U_{\ast}(z)}\left(\Exp{\Expo{-z\mathbf{k}\Gamma_{-\eta}^{1/\beta}}}-\Exp{\Expo{-z\mathbf{k}\left(2I_{\beta^2\sigma^{2}t/4}^{(\eta)}\right)^{-1/\beta}}}\right).
\end{split}
\end{equation*}
Another important identity that we will use in our arguments is the following  identity in law
\begin{equation*}
\left\{\frac{1}{I_t^{(\eta)}}, t>0\right\}\overset{(d)}{=}\left\{\frac{1}{I_t^{(-\eta)}}+2\Gamma_{-\eta}, t>0\right\},
\end{equation*}
where $\Gamma_{-\eta}$ and $I_t^{(-\eta)}$ are independent (see for instance identity (1.1) in Matsumoto and Yor \cite{MaYor}).
We also introduce
$$h(x,y)=\Expo{-\mathbf{k}x^{1/\beta}}-\Expo{-\mathbf{k}\left(x+y\right)^{1/\beta}}, \qquad x,y\geq 0.$$
Then
\begin{equation}\label{conditioned law}
\mathbb{P}_z^\ast(Z_t>0)=\frac{1}{U_{\ast}(z)}\Exp{h\left(z^{\beta}\Gamma_{-\eta},\frac{z^{\beta}}{2I_{\beta^2\sigma^{2}t/4}^{(-\eta)}}\right)},
\end{equation}
where $\Gamma_{-\eta}$ and  $I_{t}^{(-\eta)}$ are independent.

\begin{theorem}\label{supercritico }
Let $(Z_{t}, t\geq 0)$ be the supercritical stable CBBRE with index $\beta\in(0,1)$ defined by the SDE (\ref{csbpbrest}) with $Z_0=z>0$. \begin{itemize}

\item[i)] Weakly supercritical. If $\mathbf{m}\in (0,\beta\sigma^2)$, then
\begin{equation*}
\underset{t\rightarrow\infty}{\lim}t^{\frac{3}{2}} e^{\frac{\mathbf{m}^2 t}{2\sigma^2}}\mathbb{P}_z^\ast(Z_t>0)= \frac{8}{\beta^3\sigma^3\Gamma(|\eta|)U_{\ast}(z)}\int_{0}^{\infty}\int_{0}^{\infty}h(z^\beta x, z^\beta y)\phi_{|\eta|}(y) x^{-\eta-1}e^{-x} \ud x\ud y,
\end{equation*}
where
 $$\phi_{\eta}(y)=\int_{0}^{\infty}\int_{0}^{\infty}\frac{1}{\sqrt{2}\pi}\Gamma\left(\frac{\eta+2}{2}\right)e^{-y}y^{-\eta/2}u^{(\eta-1)/2}e^{-u}\frac{\sinh(\xi)\cosh(\xi)\xi}{(u+y\cosh(\xi)^{2})^{\frac{\eta+2}{2}}}\ud \xi \ud u. $$
\item[ii)] Intermediately supercritical. If $\mathbf{m}=\beta\sigma^2$, then
\begin{equation*}
\underset{t\rightarrow\infty}{\lim}\sqrt{t} e^{\beta^2\sigma^2 t/2}\mathbb{P}_z^\ast(Z_t>0)=\frac{z\mathbf{k}\sqrt{2}}{\beta^2\sigma \sqrt{\pi}U_{\ast}(z)}
\underset{n=0}{\overset{\infty}{\sum}}\frac{(-z\mathbf{k})^{n}}{n!} \Gamma\left(\frac{n+1}{\beta}+1\right)
\end{equation*}
\item[iii)] Strongly supercritical. If $\mathbf{m}>\beta\sigma^2$, then
\begin{equation*}
\underset{t\rightarrow\infty}{\lim} e^{\frac{\beta}{2} (2\mathbf{m}-\beta\sigma^2)t}\mathbb{P}_z^\ast(Z_t>0)=\frac{-z\mathbf{k}(\eta+2)}{\beta U_{\ast}(z)\Gamma(-\eta)}
\underset{n=0}{\overset{\infty}{\sum}}\frac{(-z\mathbf{k})^{n}}{n!} \Gamma\left(\frac{n+1}{\beta}-\eta-1\right)
\end{equation*}
\end{itemize}
\end{theorem}

\begin{proof}
Similarly as in the proof of Theorems  \ref{bplimitesnegativo} and \ref{bplimites}, and following the same notation as in \cite{CbMh}, we just provide the fundamental ideas of the proof. 

We first consider the weakly supercritical case (i). Note that for each $x,y>0$
\begin{equation*}
h(x,y)\leq \frac{\mathbf{k}}{\beta}x^{1/\beta-1}(y\vee y^{1/\beta}).
\end{equation*}
Since  $\Gamma_{-\eta}$  and $I_{t}^{(-\eta)}$ are independent, we deduce 
\[
\Exp{h\left(z^{\beta}\Gamma_{-\eta},\frac{z^{\beta}}{2I_{\beta^2\sigma^{2}t/4}^{(-\eta)}}\right)}=\mathbb{E}\left[g\left( \frac{z^{\beta}}{2I_{\beta^2\sigma^{2}t/4}^{(-\eta)}}\right)\right], 
\]
where $g(u):=\mathbb{E}\left[h\left(z^{\beta}\Gamma_{-\eta},u\right)\right]$. From the inequality of above, we get $g(u)\le C(u\lor u^{1/\beta})$ for $C>0$ that depends on $\mathbf{k}, \beta$ and $\eta$. 

Following step by step the proof of Lemma 4.5 in \cite{CbMh}, we can deduce that the statement also holds for our function $g$ with  $b=1$.  Actually in the proof of Lemma 4.5 in \cite{CbMh}, the authors use the inequality on their statement in order to apply the Dominated Convergence Theorem  and they split  the integral in (4.24) in \cite{CbMh} into two integrals, one over $[0,1]$ and another  over $(1,\infty)$. In our case, we can take on the integral over $[0,1]$ the function $Cu$ and on the integral over $(1,\infty)$ the function $C u^{1/\beta}$ and the result will not change. Therefore
\begin{align*}
\underset{t\rightarrow\infty}{\lim}t^{3/2}\ e^{\mathbf{m}^2 t/2\sigma^2}\mathbb{P}_z^\ast(Z_t>0)&=\underset{t\rightarrow\infty}{\lim}\frac{t^{3/2}\ e^{\eta^2 \beta^2\sigma^{2}t/8}}{U_{\ast}(z)}\mathbb{E}\left[g\left( \frac{z^{\beta}}{2I_{\beta^2\sigma^{2}t/4}^{(-\eta)}}\right)\right]\\
&=\frac{8}{\beta^3\sigma^3\Gamma(|\eta|)U_{\ast}(z)}\int_{0}^{\infty}\int_{0}^{\infty}h(z^\beta x, z^\beta y)\phi_{|\eta|}(y) x^{-\eta-1}e^{-x} \ud x\ud y,
\end{align*}
where $\phi_{|\eta|}$ is defined as in the statement of the Theorem.

In the remaining two cases we use the following inequalities, which hold by the Mean Value Theorem. Let $\epsilon>0$ then, for each $x,y\geq 0$
\begin{equation}\label{cotas}
\begin{split}
\frac{\mathbf{k}}{\beta}e^{-\mathbf{k}x^{1/\beta}}x^{1/\beta-1}y\leq h(x,y)
\leq \frac{\mathbf{k}}{\beta}e^{-\mathbf{k}x^{1/\beta}}\left((x+\epsilon)^{1/\beta-1}y+\left(\frac{x}{\epsilon}+1\right)^{1/\beta-1}y^{1/\beta}\right).
\end{split}
\end{equation}

\noindent For the intermediately supercritical case (ii), we note that $-\eta=2$. From  Lemma \ref{bpmomento} (with $p=1$) and Lemma 4.4 in \cite{CbMh} we deduce
$$\underset{t\rightarrow\infty}{\lim}\sqrt{t}e^{2t}\Exp{\frac{1}{2I_t^{(2)}}}=\frac{1}{\sqrt{2\pi}}.$$
On the other hand, from Lemma 4.5 in \cite{CbMh} with $g(u)=u^{1/\beta}$, we have
$$\underset{t\rightarrow\infty}{\lim}t^{3/2}e^{2t}\Exp{\frac{1}{(2I_t^{(2)})^{1/\beta}}}=\int_0^t g(u)\phi_{2}(u) \ud u,$$
where $\phi_2$ is defined as in the statement of  the Theorem. Therefore by the previous limits, the independence between $\Gamma_{2}$  and $I_{t}^{(2)}$, identity \eqref{conditioned law} and inequalities \eqref{cotas}, we have that for $\epsilon>0$ the following inequalities hold
\begin{equation*}
\begin{split}
\frac{z\mathbf{k}\sqrt{2}}{\beta^2\sigma \sqrt{\pi}U_{\ast}(z)}\Exp{e^{-\mathbf{k}z\Gamma_{2}^{1/\beta}}\Gamma_{2}^{1/\beta-1}}\leq \underset{t\rightarrow\infty}{\lim}\sqrt{t} e^{\beta^2\sigma^2 t/2}&\mathbb{P}_z^\ast(Z_t>0)\\
&\hspace{-1cm}\leq \frac{z^{\beta}\mathbf{k}\sqrt{2}}{\beta^2\sigma \sqrt{\pi}U_{\ast}(z)}\Exp{e^{-\mathbf{k}z\Gamma_{2}^{1/\beta}}(z^{\beta}\Gamma_{2}+\epsilon)^{1/\beta-1}}.
\end{split}
\end{equation*}
Thus our claim holds true by taking limits as $\epsilon$ goes to $ 0$. 

Finally, we use similar arguments for the strongly supercritical case (iii). Observe from Lemma \ref{bpmomento} and the identity in law by Dufresne \eqref{dufresne} that
\begin{equation}\label{eq 1}
\underset{t\rightarrow\infty}{\lim}e^{-2(1+\eta)t}\Exp{\frac{1}{2I_t^{(-\eta)}}}=\Exp{\Gamma_{-(\eta+2)}},
\end{equation}
where $\Gamma_{-(\eta+2)}$ is a Gamma r.v. with parameter $-(\eta+2)$. If $-\eta<2/\beta$, Lemma 4.5 in \cite{CbMh} imply
\begin{equation}\label{eq 2}
\underset{t\rightarrow\infty}{\lim}t^{3/2}e^{\eta^2t/2}\Exp{\frac{1}{(2I_t^{(-\eta)})^{1/\beta}}}=\int_0^{\infty}y^{1/\beta}\phi_{|\eta|}(y) \ud y,
\end{equation}
where $\phi_{|\eta|}$ is defined as in the statement of  the Theorem. If $-\eta=2/\beta$, from Lemma \ref{bpmomento} and Lemma 4.4 in \cite{CbMh}, we get
\begin{equation}\label{eq 3}
\underset{t\rightarrow\infty}{\lim}\sqrt{t}e^{2t/\beta^2}\Exp{\frac{1}{(2I_t^{(-\eta)})^{1/\beta}}}=\frac{\Gamma(1/\beta)}{\sqrt{2\pi}}.
\end{equation}
Next, if $-\eta>2/\beta$, from Lemma \ref{bpmomento} and the identity in law by Dufresne \eqref{dufresne} we get 
\begin{equation}\label{eq 4}
\underset{t\rightarrow\infty}{\lim}e^{-2t(1/\beta+\eta)/\beta}\Exp{\frac{1}{(2I_t^{(-\eta)})^{1/\beta}}}=\Exp{\Gamma_{-(\eta+2/\beta)}^{1/\beta}},
\end{equation}
where $\Gamma_{-(\eta+2/\beta)}$ is a Gamma r.v. with parameter $-(\eta+2/\beta)$. Therefore, from the independence between $\Gamma_{-\eta}$  and $I_{t}^{(-\eta)}$, inequalities \eqref{cotas} and the limits in \eqref{eq 1}, \eqref{eq 2},\eqref{eq 3} and \eqref{eq 4}, we deduce that for $\epsilon>0$, we have
\begin{equation*}
\begin{split}
\frac{z\mathbf{k}}{\beta U_{\ast}(z)}\Exp{e^{-\mathbf{k}z\Gamma_{-\eta}^{1/\beta}}\Gamma_{-\eta}^{1/\beta-1}}\Exp{\Gamma_{-(\eta+2)}}&\leq \underset{t\rightarrow\infty}{\lim} e^{-\frac{\beta}{2} (2\mathbf{m}-\beta\sigma^2)t}\mathbb{P}_z^\ast(Z_t>0)\\
&\leq \frac{z^{\beta}\mathbf{k}}{\beta U_{\ast}(z)}\Exp{e^{-\mathbf{k}z\Gamma_{-\eta}^{1/\beta}}(z^{\beta}\Gamma_{-\eta}+\epsilon)^{1/\beta-1}}\Exp{\Gamma_{-(\eta+2)}}.
\end{split}
\end{equation*}
The proof is completed once we take limits as $\epsilon$ goes to $0$.
\end{proof}

\section{The immigration case.}
In this section, we introduce  continuous state branching processes with immigration in a Brownian random
environment. In particular, this class of processes is  an extension of the Cox-Ingersoll-Ross model in  a random environment.
For simplicity, we introduce such class of processes under the assumption that the branching mechanism posses finite mean.

Recall that a CB-process with immigration (or CBI-process) is a strong Markov process taking values in $[0,\infty]$, where 0 is no longer an absorbing state. It is characterized by its branching mechanism,  
$$\psi(u)=-\alpha u+\gamma^2 u^{2}+\int_{(0,\infty)}(e^{-ux}-1+ux)\mu(\ud x),$$ and its immigration mechanism, $$\phi(u)=\mathtt{d}u+\int_0^{\infty} (1-e^{-u t})\nu(\ud t),$$ where $\alpha \in \R$, $\gamma,\mathtt{d}\geq 0$ and 
\begin{equation}\label{incondimm}
\int_0^{\infty}(x\wedge x^2)\mu(\ud x)+\int_0^{\infty}(1\wedge x)\nu(\ud x)<\infty.
\end{equation}
It is well-known that if $(Y_t, t\ge 0)$ is a process in this class, then its semi-group is characterized by
$$\Expx{x}{e^{-\lambda Y_t}}=\Expo{-xu_t(\lambda)-\int_0^t\phi(u_s)\ud s}, \qquad \mbox{for} \quad \lambda\geq 0,$$
where $u_t$ solves 
\begin{align*}
\frac{\partial u_{t}(\lambda)}{\partial t}=-\psi(u_{t}(\lambda)), \qquad u_{0}(\lambda)=\lambda.
\end{align*}
According to Fu and Li \cite{FuLi}, a CBI-process can be defined as the unique non-negative strong solution of the stochastic differential equation
\begin{align*}
Y_t=Y_0+&\int_0^t(\mathtt{d}+\alpha Y_s) \ud s+\int_0^t \sqrt{2\gamma^2 Y_s}\ud B_s\\
&\hspace{2cm}+\int_0^t\int_{(0,\infty)}\int_0^{Y_{s-}}z\widetilde{N}(\ud s,\ud z,\ud  u)+\int_0^t\int_{(0,\infty)}z M(\ud s,\ud z),
\end{align*}
where $B=(B_t,t\geq 0)$ is a standard Brownian motion, $N(\ud s,\ud z,\ud u)$ and $M(\ud s,\ud z)$ are two independent Poisson random measures with intensities $\ud s\mu(\ud z)\ud u$ and $\ud s\nu(\ud z)$, respectively, satisfying the integral condition (\ref{incondimm}) and $\widetilde{N}$ is the compensated measure of $N$. The processes $B$, $N$ and $M$ are mutually independent.


Motivated from the definition of CBBRE, we introduce a continuous state branching process with immigration in a Brownian random
environment (in short a CBIBRE) as the unique non-negative strong solution of the
stochastic differential equation
\begin{equation}\label{cbibre}\begin{split}
Z_t=&Z_0 +\int^t_0\left(\mathtt{d} +\alpha Z_s\right)\ud s +\int_0^t \sqrt{2\gamma^2 Y_s}\ud B_s+\sigma\int_0^t  Z_s\ud B^{(e)}_s\\
&\hspace{2cm}+\int_0^t\int_{(0,\infty)}\int_0^{Z_{s-}}z\widetilde{N}(\ud s,\ud z,\ud u)+\int_0^t\int_{(0,\infty)}zM(\ud s,\ud z),
\end{split}
\end{equation}
where  $\sigma>0$  and  $(B^{(e)}_t,t\geq 0)$ is a standard Brownian motion independent of $B$ and the Poisson random measures $N$ and $M$. 
Similarly as in the case with no immigration, we  define the auxiliary process
\begin{align*}
K^{(0)}_t=\sigma B^{(e)}_t+\mathbf{m}t, \qquad \textrm{for }\quad t\geq 0,
\end{align*}
where 
\[
\mathbf{m}=\alpha-\frac{\sigma^2}{2}.
\] The following Theorem provides the existence of the CBIBRE as the unique  strong solution of  (\ref{cbibre}). 

\begin{theorem}\label{CBIBRE}
The stochastic differential equation (\ref{cbibre}) has a unique non-negative strong solution. The process $Z=(Z_t, t\geq 0)$ is a Markov process and its infinitesimal generator $\mathcal{A}$ satisfies, for every $f\in C^2_b(\bar{\R}_+),$
\begin{equation}\label{bpigenerador}\begin{split}
\mathcal{A}f(x)&=\frac{1}{2}\sigma^2x^2f''(x)+\Big(x \alpha+\mathtt{d}\Big)f'(x)+x\gamma^2f''(x)\\
&+\int_{(0,\infty)}(f(x+z)-f(x))\nu(\ud z)+x\int_{(0,\infty)} \left(f(x+z)-f(x)-zf'(x)\right)\mu(\ud z).
\end{split}
\end{equation}
Furthermore, the process $Z$, conditioned on $K^{(0)}$, satisfies  the branching property and for every $t>0$
\begin{equation}\label{bpilaplacek}
\begin{split}
&\Expconx{z}{\Expo{-\lambda Z_t e^{-K^{(0)}_t}}}{K^{(0)}}\\
&\hspace{3cm}=\Expo{-zv_t(0,\lambda,K^{(0)})-\int_0^t\phi\Big(v_t(r,\lambda,K ^{(0)})e^{-K^{(0)}_r}\Big)\ud r}\quad a.s.,
\end{split}
\end{equation}
where for every $(\lambda,\delta)\in(\R_+,C(\R_+))$, $v_t: s\in[0,t]\mapsto v_t(s,\lambda,\delta)$ is the unique solution of the backward differential equation
\begin{align}\label{bpibackward}
\frac{\partial}{\partial s}v_t(s,\lambda, \delta)=e^{\delta_s}\psi_0(v_t(s,\lambda,\delta)e^{-\delta_s}),\qquad v_t(t,\lambda, \delta)=\lambda,
\end{align}
and $\psi_0(\lambda)=\psi(\lambda)+\alpha \lambda,$ for $\lambda\ge 0$.
\end{theorem}
\begin{proof}We just explain the main steps of the proof, since it follows from similar arguments as those used in Theorem 1. The existence of a strong non-negative solution of (\ref{cbibre}) is  analogous (Theorem 2.5 in \cite{DaLi}) and we omitted. The uniqueness implies the strong Markov property and by It\^o's formula it is easy to show that  the infinitesimal generator of $Z$ satisfies (\ref{bpigenerador}). 

\noindent The branching property of $Z_t$ conditioned on $K^{(0)}$, is due to the CBI-process. Let $\widetilde{Z}_{t}=Z_{t}e^{-K^{(0)}_{t}}$ and $F\in C^{1,2}(\R_+, \bar{\R}_+)$, It\^o's formula  again guarantees that $F(t,\widetilde{Z}_{t})$ conditioned on $K^{(0)}$ is a local martingale if and only if for every $t\geq 0$,
\[
\begin{split}
0&=\int_{0}^t \left(\frac{\partial}{\partial t}F(s, \widetilde{Z}_s)+\mathtt{d}\frac{\partial}{\partial x}F(s, \widetilde{Z}_s)+\gamma^2 e^{-K^{(0)}_s}\widetilde{Z}_s\frac{\partial^2}{\partial x^2}F(s, \widetilde{Z}_s) \right)\ud s  \\
&\qquad+\int_0^t\int_0^{\infty} Z_s\left( F(s, \widetilde{Z}_s+ze^{-K^{(0)}_s})-F(s, \widetilde{Z}_s)-\frac{\partial}{\partial x}F(s, \widetilde{Z}_s)ze^{-K^{(0)}_s}\right)\mu(\ud z)\ud s\\
&\qquad+\int_0^t\int_0^{\infty} \left( F(s, \widetilde{Z}_s+ze^{-K^{(0)}_s})-F(s, \widetilde{Z}_s)\right)\nu(\ud z)\ud s.
\end{split}
\]
We choose $F(s,x)=\Expo{-xv_t(s,\lambda,K^{(0)} )-\int_s^t\phi(v_t(r,\lambda,K^{(0)} )e^{-K^{(0)}_r})\ud r}$ where $v_t(s,\lambda,K^{(0)})$ is differentiable with respect to the variable $s$, non-negative and such that  $v_t(t,\lambda, K^{(0)})=\lambda$ for all $\lambda\geq 0$. We observe that $F$ is bounded, therefore, $(F(s,\widetilde{Z}_s),s\leq t)$ conditioned on $K^{(0)}$ is a martingale if and only if 
\begin{align*}
\frac{\partial}{\partial s}v_t(s,\lambda,K^{(0)})&=\gamma^2(v_t(s,\lambda,K^{(0)}))^2e^{-K^{(0)}_s}\\
&+ e^{K^{(0)}_s}\int_{0}^{\infty}\left(e^{-e^{-K^{(0)}_s}v_t(s,\lambda, K^{(0)})z}-1+e^{-K^{(0)}_s}v_t(s,\lambda, K^{(0)})z\right)\mu(\ud z),
\end{align*}
which is equivalent to the backward differential equation given  in  (\ref{bpibackward}). The existence and unicity of $v_t$ follows from the Picard-Lindelöf theorem. 

\noindent Therefore, the process $\left(\Expo{-\widetilde{Z}_{s}v_{t}(s,\lambda, K^{(0)})}, 0\leq s\leq t\right)$ conditioned on $K^{(0)}$ is a martingale, and hence 
$$\Expconx{z}{\Expo{-\lambda \widetilde{Z_t}}}{K^{(0)}}=\Expo{-zv_t(0,\lambda,K^{(0)})-\int_0^t\phi(v_t(r,\lambda,K^{(0)} )e^{-K^{(0)}_r})\ud r}.$$
\end{proof}

We finish this section with some examples of  CBIBRE, all of them related to the stable CBBRE. \\

{\bf Example 4. (Stable CBIBRE)}  Here we assume that the branching and immigration  mechanisms  are of the form $\psi(\lambda)=-\alpha\lambda+c\lambda^{\beta+1}$  and $\phi(\lambda)=\kappa\lambda^{\beta}$, where $\beta\in (0,1)$, $c, \kappa>0$  and  $a\in \mathbb{R}$.  Hence,   the stable CBIBRE-process is given as the unique non-negative strong solution of the stochastic differential equation
\[
\begin{split}
Z_t=&Z_0 +\int^t_0 \alpha Z_s\ud s +\sigma\int_0^t  Z_s\ud B^{(e)}_s+\int_0^t\int_0^{\infty}\int_0^{Z_{s-}}z\widetilde{N}(\ud s,\ud z,\ud u)+\int_0^t\int_0^{\infty}z M(\ud s,\ud z),
\end{split}
\]
where $B^{(e)}$ is a Brownian motion, and $N$ and $M$ are two independent Poisson random measures with intensities
\[
\frac{c \beta(\beta+1)}{\Gamma(1-\beta)} \frac{1}{z^{2+\beta}}\ud s\ud z\ud u \qquad\textrm{and}\qquad \frac{\kappa \beta}{\Gamma(1-\beta)}\frac{1}{z^{1+\beta}} \ud s\ud z.
\]
Its associated infinitesimal generator $\mathcal{A}$ satisfies, for every $f\in C^2(\R_+),$
\begin{equation}\label{stablebpigenerador}\begin{split}
\mathcal{A}f(x)&=\frac{1}{2}\sigma^2x^2f''(x)+x\alpha f'(x)+\frac{\kappa \beta}{\Gamma(1-\beta)}\int_{(0,\infty)}(f(x+z)-f(x))\frac{ \ud z}{z^{1+\beta}}\\
&\hspace{4cm}+x\frac{c \beta(\beta+1)}{\Gamma(1-\beta)} \int_{(0,\infty)} \left(f(x+z)-f(x)-zf'(x)\right)\frac{\ud z}{z^{2+\beta}}.
\end{split}
\end{equation}
From (\ref{bpilaplacek}) we get the following a.s. identity
\begin{equation}\label{stablecbi}
\begin{split}
\Expconx{z}{\Expo{-\lambda Z_t e^{-K^{(0)}_t}}}{K^{(0)}}=
&\Expo{-z\left(\beta c\int_0^te^{-\beta K^{(0)}_s}\ud s+\lambda^{-\beta}\right)^{-1/\beta}}\\
&\times\Expo{ -\frac{\kappa}{\beta c}\ln\left(\beta c\lambda^\beta \int_0^te^{-\beta K^{(0)}_s}\ud s+1\right)}.
\end{split}
\end{equation}
If we take limits as $z$ goes to $ 0$, we deduce that the entrance law at $0$ of the process $(Z_te^{-K_t^{(0)}}, t\ge 0)$ satisfies
\[
\Expconx{0}{\Expo{-\lambda Z_t e^{-K^{(0)}_t}}}{K^{(0)}}=\Expo{ -\frac{\kappa}{\beta c}\ln\left(\beta c\lambda^\beta \int_0^te^{-\beta K^{(0)}_s}\ud s+1\right)}.
\]
If we take  limits as $\lambda$ goes to $\infty $ in (\ref{stablecbi}),  we obtain 
$$\mathbb{P}_z\Big(Z_t>0\Big|K^{(0)}\Big)=1, \qquad \textrm{for }\quad z\ge 0.$$
Similarly if we take limits as $\lambda$ goes to $0$ in (\ref{stablecbi}), we deduce
$$\mathbb{P}_z\Big(Z_t<\infty\Big|K^{(0)}\Big)=1, \qquad \textrm{for }\quad z\ge 0.$$
In other words, the stable CBIBRE is conservative and positive at finite time a.s.

An interesting question is to study the long-term behaviour of the  stable CBIBRE. Now, if we take limits as $t$ goes to $\infty$ in (\ref{stablecbi}), we deduce that when $\mathbf{m}>0$, 
\[
\begin{split}
\mathbb{E}_z\left[\Expo{-\lambda \lim_{t\to\infty}Z_t e^{-K^{(0)}_t}}\right]=&\mathbb{E}\Bigg[\Expo{-z\left(\beta c\int_0^\infty e^{-\beta K^{(0)}_s}\ud s+\lambda^{-\beta}\right)^{-1/\beta}}\\
&\times\Expo{ -\frac{\kappa}{\beta c}\ln\left(\beta c\lambda^\beta \int_0^\infty e^{-\beta K^{(0)}_s}\ud s+1\right)}\Bigg],
\end{split}
\]
where we recall that $ \int_0^\infty e^{-\beta K^{(0)}_s}\ud s$ has the same law as  $\Big(2\Gamma_{-\eta}\Big)^{-1}$, with  $\Gamma_{v}$ is a Gamma r.v. with parameter $v$. In other words, $Z_t e^{-K^{(0)}_t}$ converges in distribution to a r.v. whose Laplace transform is given by the previous identity.

If $\mathbf{m}\le 0$, we deduce 
\[
\lim_{t\to\infty}Z_t e^{-K^{(0)}_t}=\infty,\qquad \mathbb{P}_z-\textrm{a.s.}
\]
We observe that when $\mathbf{m}= 0$, the process $K^{(0)}$ oscillates implying that
\[ 
\lim_{t\to\infty}Z_t =\infty,\qquad \mathbb{P}_z-\textrm{a.s.}
\]
\\

{\bf Example 5. (Stable CBBRE Q-process)}. Here, we  assume $\mathbf{m}\le -\sigma^2$. From Proposition \ref{qprocess}, we deduce that the form of the infinitesimal generator of the stable CBBRE Q-process, here denoted by $\mathcal{L}^\natural$, satisfies for $f\in \textrm{Dom}(\mathcal{L})$
\[
\begin{split}
\mathcal{L}^\natural f(z)&=\mathcal{L}f(z)+z^2\sigma^2f^\prime(z)\frac{U^\prime(z)}{U(z)}\\
&\qquad+\frac{c \beta(\beta+1)}{\Gamma(1-\beta)}\frac{z}{U(z)}\int_0^\infty \Big(f(y+z)-f(z)\Big)\Big(U(z+y)-U(z)\Big)\frac{\ud y}{y^{2+\beta}},
\end{split}
\]
where $U$ was defined as
\[
U(z)=
\left\{
\begin{array}{ll}
\vspace{0.2cm}z \displaystyle\frac{\sqrt{2}}{\sqrt{\pi}\beta\sigma}\mathbf{k}\Gamma\left(\frac{1}{\beta}\right) & \mbox{if } \mathbf{m}=-{\sigma^2},\\
\vspace{0.2cm} z \mathbf{k}\displaystyle\frac{\Gamma(\eta-1/\beta)}{\Gamma(\eta-2/\beta)} & \mbox{if }  \mathbf{m}<-{\sigma^2}.
\end{array}
\right.
\]
Replacing the form of $U$ in the infinitesimal generator $\mathcal{L}^\natural $ in both cases, we get
\[
\begin{split}
\mathcal{L}^\natural f(z)&=\frac{1}{2}\sigma^2z^2f''(z)+(\alpha+\sigma^2)zf'(z)+\frac{c \beta(\beta+1)}{\Gamma(1-\beta)}\int_0^\infty \Big(f(z+x)-f(x)\Big)\frac{\ud z}{x^{1+\beta}}\nonumber\\
&\qquad+z\frac{c \beta(\beta+1)}{\Gamma(1-\beta)} \int_{(0,\infty)} \left(f(x+z)-f(z)-xf'(z)\right)\frac{\ud x}{x^{2+\beta}}.
\end{split}
\]
From the form of the infinitesimal generator  of the stable CBIBRE given in (\ref{stablebpigenerador}), we deduce that the stable CBBRE Q-process is a stable CBIBRE  with branching and immigration mechanisms given by 
\[
\psi(\lambda)=-(\sigma^2+\alpha)\lambda +c\lambda^{1+\beta} \qquad\textrm{and} \qquad \phi(\lambda)=c(\beta+1)\lambda^{\beta},
\]
and random environment  $(\sigma B_t^{(e)}, t\ge0)$. \\


\noindent \textbf{Acknowledgements}\\

\noindent SP thanks the  Department of Mathematical Sciences of the University of Bath where part of this work was done.  SP and JCP acknowledge support from the  Royal Society  and  SP also acknowledge support from  CONACyT-MEXICO Grant 351643. Both authors are grateful to  anonymous referees
for their suggestions which helped improved the presentation of this paper.

\end{document}